\documentclass[11pt]{article}
\usepackage[utf8]{inputenc} \usepackage[T1]{fontenc}    \usepackage[english]{babel}

\usepackage{amsfonts}
\usepackage{nicefrac}       \usepackage{microtype}      \usepackage{lmodern}
\usepackage{amssymb,amsmath,amsthm}
\usepackage{stmaryrd}
\SetSymbolFont{stmry}{bold}{U}{stmry}{m}{n}
\usepackage{bm}
\usepackage{latexsym}
\usepackage{xcolor}

\usepackage{enumerate}
\usepackage{verbatim}
\usepackage{booktabs}       

\usepackage{url}            \usepackage[colorlinks=true]{hyperref}
\usepackage[numbers,sort&compress,square,comma]{natbib}
\usepackage[capitalise,noabbrev]{cleveref}

\usepackage{parskip}
\usepackage{geometry}
\usepackage[short]{optidef}
\usepackage{algorithm} 
\usepackage{algorithmic}  
\usepackage[algo2e]{algorithm2e} 
\usepackage{thmtools}
\usepackage{authblk}

\declaretheorem{theorem}
\declaretheorem{corollary}
\declaretheorem{lemma}
\declaretheorem{proposition}

\declaretheorem{conjecture}

\declaretheorem[style=definition]{remark}

\crefname{assumption}{Assumption}{Assumptions}
\newcommand{\floor}[1]{\ensuremath{\left\lfloor #1 \right\rfloor}}

\newcommand{\abs}[1]{\ensuremath{\left\lvert #1 \right\rvert}}

\newcommand{\by}{\times}
 
\newcommand{\norm}[1]{\ensuremath{\left\lVert #1 \right\rVert}}
\newcommand{\ip}[1]{\ensuremath{\left\langle #1 \right\rangle}}

\newcommand{\set}[1]{\left\{#1\right\}}

\def\B{{\mathbb{B}}}

\def\O{{\mathbb{O}}}

\def\R{{\mathbb{R}}}
\def\S{{\mathbb{S}}}
\def\T{{\mathbb{T}}}

\def\X{{\mathbb{X}}}

\def\Z{{\mathbb{Z}}}

\def\bS{{\mathbf{S}}}

\def\cB{{\cal B}}
\def\cC{{\cal C}}

\def\cL{{\cal L}}

\def\cR{{\cal R}}

\def\cT{{\cal T}}
\def\cU{{\cal U}}
\def\cV{{\cal V}}

\DeclareMathOperator*{\argmax}{arg\,max}

\DeclareMathOperator{\rank}{rank}
\DeclareMathOperator{\Diag}{Diag}
\DeclareMathOperator{\diag}{diag}
\DeclareMathOperator{\tr}{tr}

\DeclareMathOperator{\inter}{int}

\DeclareMathOperator{\conv}{conv}

\newcommand{\framedheader}[2]{
  \framebox[\textwidth]{
    \vbox{
      \vspace{2mm}
      \hbox to \textwidth {\hspace{1em}\today \hfill #1\hspace{1em}}
      \vspace{4mm}
      \hbox to \textwidth {\hfill \Large{#2} \hfill}
      \vspace{2mm}
    }
  }
  \vspace*{4mm}
}

\usepackage{soul}

\let\O\undefined
\DeclareMathOperator{\O}{O}
\DeclareMathOperator{\SO}{SO}
\newcommand{\Bop}{\B_\textup{op}}

 \usepackage{tikz}
\usepackage{multirow}
\newcommand{\sut}{\cT}

\newcommand{\change}[1]{#1}
\newcommand{\changetwo}[1]{#1}
\newcommand{\PP}{\textup{PP}}
\DeclareMathOperator{\spin}{Spin}
\DeclareMathOperator{\dist}{dist}
\DeclareMathOperator{\str}{str}
\DeclareMathOperator{\sign}{sign}

\usepackage{pgfplots}
\pgfplotsset{compat=1.18}
\usepackage{graphicx}
\usepackage{import}
\usepackage[font={small,it}]{caption}

\newcommand{\ije}[1]{}

\begin{document}
\title{Hidden convexity, optimization, and algorithms on rotation matrices}
\author[1]{Akshay Ramachandran}
\author[2]{Kevin Shu}
\author[1,3]{Alex L. Wang}
\affil[1]{Centrum Wiskunde \& Informatica, Amsterdam, The Netherlands}
\affil[2]{George Institute of Technology, Atlanta, GA}
\affil[3]{Purdue University, West Lafayette, IN}
\maketitle

\begin{abstract}

This paper studies 
hidden convexity properties associated with
constrained optimization problems over the set of rotation matrices $\SO(n)$.
Such problems are nonconvex due to the constraint $X\in\SO(n)$. Nonetheless, we show that certain linear images of $\SO(n)$ are convex, opening up the possibility for convex optimization algorithms with provable guarantees for these problems.
Our main technical contributions show that any two-dimensional image of $\SO(n)$ is convex and that the projection of $\SO(n)$ onto its strict upper triangular entries is convex. These results allow us to construct exact convex reformulations for constrained optimization problems over $\SO(n)$ with a single constraint or with constraints defined by low-rank matrices.
Both of these results are \change{maximal} in a formal sense. \end{abstract}

\section{Introduction}
This paper studies a general class of optimization problems over rotations and orthogonal bases.
This class of problems covers applications such as the point registration problem in computer graphics \cite{ConvexPointRegistration, FunctionalMatching}, Wahba's problem of satellite attitude determination \cite{wahba1965least}, spacecraft orientation \cite{lee2011spacecraft}, and obstacle avoidance in robotics \cite{chen2020active}.
The main goal of this paper is show that in certain cases of interest, we can produce natural \emph{convex relaxations} that exactly recover the optimal solutions for such problems.

Recall, $\O(n)$ is the set of \change{orthogonal matrices} in $\R^n$, or more explicitly,
\[
    \O(n) \coloneqq \{X \in \R^{n\times n} : X^{\intercal}X = I\}.
\]
On the other hand, $\SO(n)$ is the set of (orientation-preserving) rotations on $\R^n$, defined as
\[
    \SO(n) \coloneqq \set{X \in \R^{n\times n} :\, \begin{array}{l}
        X^\intercal X = I\\
        \det(X) = 1
    \end{array}}.
\]
\change{These sets are groups under matrix multiplication, refered to as the orthogonal and special orthogonal groups, respectively.}

We consider optimization problems of the form
\begin{align}
    \label{eq:general_problem_format}
    \sup_{X\in\SO(n)}\set{\ip{A,X}:\, \cB(X) \in \cC},
\end{align}
and their $\O(n)$ counterparts.
Here, the objective is a \textit{linear} function defined by $A\in\R^{n\by n}$, and the constraint is defined by a \textit{linear} operator $\cB:\R^{n\by n}\to\R^m$ and some \textit{convex} set $\cC\subseteq \R^m$.
The notation $\ip{A,B}$ denotes the trace inner product $\tr(A^\intercal B)$.
We will also study feasibility variants of \eqref{eq:general_problem_format} where the goal is to identify an $X\in\SO(n)$ or $X\in\O(n)$ satisfying $\cB(X)\in\cC$, or to declare that no such $X$ exists.
We give additional motivation for these problems in \cref{subsec:motivation}, where we discuss constrained versions of Wahba's problem and point registration~\cite{ConvexPointRegistration,wahba1965least}.

Problems of the form \eqref{eq:general_problem_format} are ostensibly \emph{nonconvex} due to the constraint $X\in\SO(n)$ or $X\in\O(n)$.
Nevertheless, we will show that certain families of such problems admit exact convex reformulations.
To achieve this, our main technical contributions show that \emph{the images of $\SO(n)$ or $\O(n)$ under certain linear maps are convex}.
Such results---showing that certain transformations of nonconvex sets are convex---are often referred to as hidden convexity results and enable the application of convex optimization algorithms to nonconvex problems~\cite{xia2020survey,polik2007survey}.

To see how such a result might be useful in solving problems of the form \eqref{eq:general_problem_format}, suppose that $\cL:\R^{n\by n}\to\R^{1+m}$ is the linear map
\[
    \cL(X) \coloneqq \begin{pmatrix}
    \ip{A,X}\\
    \cB(X)
    \end{pmatrix}.
\]
If the image of $\SO(n)$ under $\cL$ is convex, then we would have that $\cL(\SO(n)) = \conv(\cL(\SO(n))) = \cL(\conv(\SO(n)))$.
Here, $\conv(\cdot)$ represents the convex hull of the argument.
In this case, it would then follow that
\begin{align*}
    \sup_{X\in\SO(n)}\set{\ip{A,X}:\, \cB(X)\in\cC} = \sup_{X\in\conv(\SO(n))}\set{\ip{A,X}:\,\cB(X)\in\cC}.
\end{align*}
In other words, convexity of the image $\cL(\SO(n))$ implies that the convex relaxation of \eqref{eq:general_problem_format} that simply replaces $\SO(n)$ with $\conv(\SO(n))$ is exact. This exactness is in terms of objective value, however we will also see how to numerically recover an actual optimizer in $\SO(n)$ or $\O(n)$ in the settings we consider.
Note that the convexity of $\cB(\SO(n))$ alone does not imply exactness (see for example \cref{fig:b_convex_not_sufficient}).
Similar results can be derived for $\O(n)$ and/or feasibility variants of \eqref{eq:general_problem_format} under corresponding convexity results.

\begin{figure}
    \centering
    \def\svgscale{0.5}
\begingroup%
\makeatletter%
\providecommand\color[2][]{%
  \errmessage{(Inkscape) Color is used for the text in Inkscape, but the package 'color.sty' is not loaded}%
  \renewcommand\color[2][]{}%
}%
\providecommand\transparent[1]{%
  \errmessage{(Inkscape) Transparency is used (non-zero) for the text in Inkscape, but the package 'transparent.sty' is not loaded}%
  \renewcommand\transparent[1]{}%
}%
\providecommand\rotatebox[2]{#2}%
\newcommand*\fsize{\dimexpr\f@size pt\relax}%
\newcommand*\lineheight[1]{\fontsize{\fsize}{#1\fsize}\selectfont}%
\ifx\svgwidth\undefined%
  \setlength{\unitlength}{272.00033353bp}%
  \ifx\svgscale\undefined%
    \relax%
  \else%
    \setlength{\unitlength}{\unitlength * \real{\svgscale}}%
  \fi%
\else%
  \setlength{\unitlength}{\svgwidth}%
\fi%
\global\let\svgwidth\undefined%
\global\let\svgscale\undefined%
\makeatother%
\begin{picture}(1,0.72287829)%
  \lineheight{1}%
  \setlength\tabcolsep{0pt}%
  \put(0,0){\includegraphics[width=\unitlength,page=1]{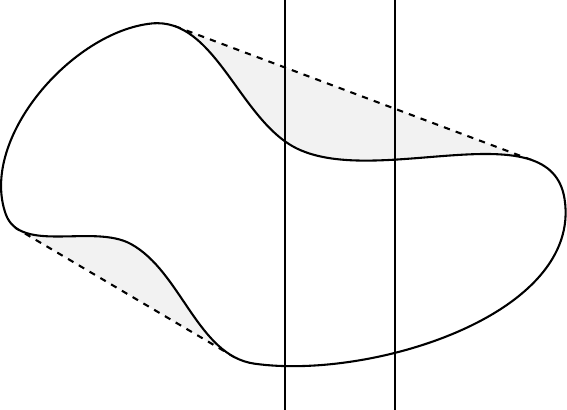}}%
  \put(0.09705983,0.39548251){\color[rgb]{0,0,0}\makebox(0,0)[lt]{\lineheight{1.25}\smash{\begin{tabular}[t]{l}$S$\end{tabular}}}}%
  \put(0.79386236,0.52006677){\color[rgb]{0,0,0}\makebox(0,0)[lt]{\lineheight{1.25}\smash{\begin{tabular}[t]{l}$\conv(S)$\end{tabular}}}}%
\end{picture}%
\endgroup%

    \caption{Consider the set $S\subseteq\R^2$ (white region with solid boundary) and its convex hull $\conv(S)$. The projection of $S$ onto the horizontal axis is convex. On the other hand, the maximum vertical direction achieved by a point $x\in S$ between the two vertical lines differs from the maximum vertical direction achieved by a point $x\in\conv(S)$ between the two vertical lines.}
    \label{fig:b_convex_not_sufficient}
\end{figure}

The convex hulls of both $\O(n)$ and $\SO(n)$ can be described via linear matrix inequalities (LMIs). The first fact is well-known while the latter fact is due to \citet{saunderson2015semidefinite}. For ease of reference, we collect both facts in the following proposition.
\begin{proposition}[Classical/\cite{saunderson2015semidefinite}]
\label{thm:convSOn}
The convex hull of $\O(n)$ is equal to the operator norm ball and can be written as 
\begin{align*}
    \conv(O(n)) = \Bop(n)=\set{X\in\R^{n\by n}:\, \begin{pmatrix}
    I & X\\X^\intercal & I
    \end{pmatrix}\succeq 0}.
\end{align*}
There exist symmetric matrices $A_{i,j}\in\S^{2^{n-1}}$ indexed by $(i,j)\in[n]^2$ such that
\begin{align*}
    \conv(\SO(n))=\set{X\in\R^{n\by n}:\, \sum_{i=1}^n\sum_{j=1}^nX_{i,j}A_{i,j}\preceq I_{2^{n-1}}}.
\end{align*}
\end{proposition}
In light of this fact, the convex relaxations we consider in the $\O(n)$ setting can be directly solved with semidefinite programs~(SDPs), assuming that $\cC$ is itself efficiently SDP-representable.
In contrast, the convex relaxations we consider in the $\SO(n)$ setting result in exponentially sized SDPs.
For this reason, we also offer new algorithms for optimizing over $\conv(\SO(n))$ in the settings we consider that do not rely on the description of $\conv(\SO(n))$ given in \cref{thm:convSOn}.

\subsection{Motivation}
\label{subsec:motivation}
We first discuss \emph{constrained} variants of Wahba's problem.
To set up Wahba's problem, imagine that a satellite in space wants to determine its relative rotation (with respect to a reference rotation) given the observed direction of some number of far-away stars (or other objects).

Formally, we are given a set of (unit) vectors $v_1 ,\dots, v_k \in \R^3$, corresponding to the known directions of the $k$ stars in the reference rotation, and (unit) vectors $u_1 ,\dots, u_k \in \R^3$, corresponding to the observed directions of the $k$ stars in the satellite's frame.
Our goal is to find a rotation minimizing the observation error
\begin{align*}
    \min_{X\in\SO(3)} \sum_{i=1}^k \norm{X u_i - v_i}_2^2.
\end{align*}
In \cite{wahba1965least}, it was observed that this is equivalent to a linear optimization problem over $\SO(3)$
\begin{align*}
    \max_{X\in\SO(3)} \ip{\sum_{i=1}^k u_i v_i^{\intercal},\, X},
\end{align*}
and that this problem can be solved in turn using a single SVD computation. \change{We will formalize this in \cref{lem:strcompute}.}

Now, suppose we are given additional information about the true rotation $X^*$. We will incorporate this additional information as hard constraints into Wahba' problem to get a constrained optimization problem over $\SO(3)$.

For example, we may
know that the true rotation $X^*$ is within some angle, $\delta$, of another rotation $X_0\in\SO(3)$.
In this case, we would need to solve the problem\footnote{\changetwo{Any $X \in \SO(3)$ can be thought of as a rotation of some angle $\delta$ around some axis, and this angle satisfies the equation $\tr(X) = 1+2\cos(\delta)$, which can be seen by considering the eigenvalues of $X$. Here, we say $X$ is at an angle $\delta$ from $Y$ if $Y^{\intercal}X$ is of angle $\delta$.}}
\begin{align*}
    \max_{X\in\SO(3)}\set{\ip{\sum_{i=1}^k u_i v_i^{\intercal}, X}:\, \langle X_0, X\rangle \ge 1 + 2 \cos(\delta)}.
\end{align*}
\cref{thm:oneconnected} below implies that this problem has the same optimal value as its convex relaxation.
We will additionally show that
the optimum value of this problem can be efficiently computed using convex optimization techniques, even for $n\geq 3$.

As a second example, we may have a few high-fidelity observations, in which case we could introduce constraints for those observations:
\begin{align*}
    \ip{X u_i, v_i} = \ip{u_iv_i^\intercal, X} \geq \cos(\delta_i).
\end{align*}
\cref{thm:opt_so_with_sut} below implies that feasibility problems with at most $n - 1$ such constraints and certain optimization problems over such constraints can be solved efficiently using convex optimization techniques. 

High-dimensional settings, where $n\geq 3$, have found use in modeling nonlinear transformations on manifolds~\cite{FunctionalMatching}.
In this setting, the goal is to learn a nonlinear transformation mapping one manifold to another based on given point--point correspondences.
\citet{FunctionalMatching} suggest modeling this problem as that of finding an orthogonal (linear) transformation in the \emph{space of functions} on the manifolds.\footnote{The function spaces are endowed with bases corresponding to (possibly a truncated set of) eigenfunctions of the Laplace--Beltrami operators.}
Note, this space of functions may be high dimensional.
In this function space, point--point correspondence constraints or symmetry constraints can be naturally expressed as linear constraints on the linear transformation.
Additional desirable properties of the nonlinear transformation can be encoded as an orthogonality constraint.
Thus, this problem has a natural interpretation as an optimization problem of the form \eqref{eq:general_problem_format} over $\O(n)$.

\subsection{Statement of Results}
Our main contributions show that certain linear images of $\SO(n)$ are convex and that certain problems of the form \eqref{eq:general_problem_format} and its variants can be solved efficiently. We give an overview of these results in the order of the sections they appear in; see also \cref{tab:summary}.
\change{For our algorithms, we assume we are able to perform basic arithmetic operations (including taking square roots) exactly.}

\begin{table}
\centering
\begin{tabular}{|l|ll|}
\hline
& Hidden convexity & Algorithms\\\hline
Diag.\ constraints & \cite[Theorem 8]{horn1954doubly} & \cref{thm:diagonal}\\
One constraint & \cref{thm:oneconnected,thm:twoconvex} & \cref{thm:2dAlgo}\\
SUT constraints & \cref{thm:opt_so_with_sut,thm:hidden_convexity_so_with_sut,thm:feasibility} & \cref{thm:ut_structure_dt}\\\hline
\end{tabular}

\caption{Summary of our hidden convexity results and algorithms for constrained optimization over $\SO(n)$.
We present accompanying results (\cref{thm:obstructions}) showing that our hidden convexity results are each ``maximal'' in certain senses \change{(see \Cref{sec:obstructions})}.}
\label{tab:summary}
\end{table}

\subsubsection{Feasibility problems on $\SO(n)$ with diagonal constraints}

A classical theorem of Horn, \cite[Theorem 8]{horn1954doubly}, gives a first example of a hidden convexity result on $\SO(n)$.
\changetwo{Recall, the parity polytope is defined as
\begin{align*}
    \PP_n \coloneqq \conv\set{x\in\set{\pm 1}^n:\, \prod_{i=1}^n x_i = 1}.
\end{align*}}
\begin{theorem}
    \label{thm:diagconv}
    Let $\textrm{diag}:\R^{n\by n}\to\R^n$ map a square matrix to its diagonal elements, then $\diag(\SO(n))$ is convex (in fact, polyhedral) and its image is given by the \emph{parity polytope} $\PP_n$.
\end{theorem}
\changetwo{In \cref{app:separation_pp}, we give} efficient separation and optimization oracles for $\PP_n$. \change{We give a definition of a separation oracle in \cref{sec:son_one_constraint}.}
It follows that feasibility problems on $\SO(n)$ with convex constraints on the diagonal elements, i.e., given convex $\cC\subseteq \R^n$, decide the feasibility of 
\begin{align}
    \label{eq:diagX_feas}
    \set{X\in\SO(n):\, \diag(X)\in\cC},
\end{align}
can be decided by testing the feasibility of $\PP_n\cap \cC$.

In \cref{sec:feas_diag}, we complete this picture by showing how to construct an element of \eqref{eq:diagX_feas} given $d\in \PP_n\cap \cC$.
\begin{restatable}{theorem}{thmdiagonal}
    \label{thm:diagonal}
    Given $d\in\PP_n$, it is possible to construct $X\in\SO(n)$ satisfying $\diag(X) = d$ in time $O(n^2)$.
\end{restatable}

\subsubsection{Optimization on $\SO(n)$ with one constraint}
\label{subsubsec:results_one_constraint}
The main result of \cref{sec:son_one_constraint} is that the intersection of $\SO(n)$ with any codimension-one affine space is connected.
\begin{restatable}{theorem}{thmoneconnected}
    \label{thm:oneconnected}
    Let $n\geq 3$, $A \in \R^{n\times n}$, and $c\in\R$. Then, the set $\SO(n) \cap \set{X\in\R^{n\by n}:\, \ip{A,X} = c}$ is connected.
\end{restatable}
A corollary of \cref{thm:oneconnected}, \change{which we prove in \cref{sec:son_one_constraint}}, is the following:
\begin{restatable}{corollary}{thmtwoconvex}
    \label{thm:twoconvex}
    Let $n\geq 3$ and let $\pi : \SO(n) \rightarrow \R^2$ be linear, then $\pi(\SO(n))$ is convex.
\end{restatable}
This follows as a set is convex if and only if its intersection with any one-dimensional affine subspace (i.e., a line) is connected.
\cref{thm:oneconnected} may be of importance in its own right\changetwo{, as for some hyperplanes $H$, it may be the case that the set $\SO(n)\cap H$ is a \emph{connected} submanifold of $\SO(n)$, in which case it may be possible to use Riemannian optimization methods to optimize functions on this set.}

While these results imply that it is possible to use convex optimization techniques to solve problems of the form
\begin{equation}
    \max_{X \in \SO(n)} \left\{\ip{A, X} : \ip{B, X} \in [a,b]\right\}\label{eq:ndmax}
\end{equation}
(for example by replacing $\SO(n)$ with $\conv(\SO(n))$), it is not clear how to do so \emph{efficiently} when $n$ is large.
This is because \Cref{thm:convSOn} only guarantees an exponentially sized LMI representation for $\conv(\SO(n))$.

To address this issue, we give an efficient algorithm for problems of the form \eqref{eq:ndmax} based on running the ellipsoid algorithm on the two-dimensional image of $\SO(n)$.
It is noteworthy that because we use the ellipsoid algorithm in a constant-dimensional space, we do not face the infamously high running times of the ellipsoid method in high-dimensional spaces.

\begin{restatable}{theorem}{thmtwodalgo}
    \label{thm:2dAlgo}
Let $n\geq 3$, $A,\, B\in\R^{n\by n}$ with $\norm{A}_{\tr} = \norm{B}_{\tr} = 1$. 
    Here $\|\cdot\|_{\tr}$ is the trace norm, defined formally in \Cref{sec:prelim}.
    Let $X^*$ be the optimal solution to \eqref{eq:ndmax}. We can compute $\ip{A,X^*}$ and $\ip{B,X^*}$ within an additive error of $\epsilon$ in time 
    \[
        O\left( n^3\log\left(\frac{1}{\epsilon}\right)^2\right).
    \]
     Here, $n^3$ is the time complexity of computing the SVD of an $n\times n$ matrix.

    Moreover, we will return $\alpha, \beta \in \R$ so that 
    $|\alpha| + |\beta| = 1$ and
    \[
        \ip{\alpha A + \beta B, X^*} + \epsilon \ge \max \{\ip{\alpha A + \beta B, X} : X \in \SO(n)\}.
    \]
\end{restatable}

\begin{remark}
Let $\alpha,\,\beta$ denote the quantities returned in \cref{thm:2dAlgo}. While \Cref{thm:2dAlgo} does not directly return a minimizer of \eqref{eq:ndmax}, we believe that
any element of 
\begin{align*}
    \argmax_{X\in\SO(n)}\ip{\alpha A + \beta B, X}
\end{align*}
should be a good approximation of a true minimizer under mild conditions. Such an element can be computed from $\alpha A + \beta B$ in the time of a single SVD decomposition.
Analyzing this procedure is outside the scope of the current paper and we leave this question for future work.
\end{remark}

\subsubsection{Feasibility and optimization on $\SO(n)$ with strictly upper triangular constraints}
The last class of constraints we consider are constraints on the strictly upper triangular (SUT) entries of $X\in\SO(n)$.
Our main result in this direction shows that not only is the projection of $\SO(n)$ onto its SUT entries convex, but furthermore, it is possible to optimize certain linear functions subject to convex constraints on the SUT entries using convex optimization.

Let $\pi_\sut(X) = (X_{ij})_{i<j}\in\R^{\binom{n}{2}}$ denote the projection of $X$ onto its SUT entries (i.e., those entries $X_{ij}$ such that $i<j$).
We will consider constraining the value of $\pi_\sut(X)$ for $X\in\SO(n)$ and then optimizing a linear function over this set.
Let $A\in\R^{n\by n}$ and let $\cC$ be a nonempty closed convex subset of $\pi_\sut(\Bop(n))$.
Recall $\Bop(n)$ is the operator norm ball and is the same as $\conv(\O(n))$ by \Cref{thm:convSOn}.
We consider the problems
\begin{align}
    &\sup_{X\in\SO(n)}\set{\ip{A,X}:\, \pi_\sut(X)\in \cC} \label{eq:opt_over_so_with_sut}\\
    &\qquad\leq \sup_{X\in\O(n)}\set{\ip{A,X}:\, \pi_\sut(X)\in \cC}\label{eq:opt_over_o_with_sut}\\
    &\qquad\leq 
    \max_{X\in\Bop(n)}\set{\ip{A,X}:\, \pi_\sut(X)\in \cC}. \label{eq:opt_over_bop_with_sut}
\end{align}

\change{The $\max$ in \eqref{eq:opt_over_bop_with_sut} will be justified in \cref{sec:sut_constraints}.}
Our main result on this topic is:
\begin{restatable}{theorem}{thmoptsosut}
    \label{thm:opt_so_with_sut}
    Let $A\in\R^{n\by n}$ be a diagonal matrix and let $\cC\subseteq\pi_\sut(\Bop(n))$ be a nonempty closed convex set. Then, equality holds between \eqref{eq:opt_over_o_with_sut} and \labelcref{eq:opt_over_bop_with_sut}.
    If additionally $\det(A)\geq 0$, then equality holds between \labelcref{eq:opt_over_so_with_sut,eq:opt_over_o_with_sut,eq:opt_over_bop_with_sut}.
\end{restatable}
\change{A corollary} of \cref{thm:opt_so_with_sut}\change{, which will be proved in \cref{sec:sut_constraints},} is the following:
\begin{restatable}{corollary}{thmfeasibility}
    \label{thm:feasibility}
It holds that $\pi_\sut(\SO(n)) = \pi_\sut(\O(n)) = \pi_\sut(\Bop(n))$. In particular, all three sets are convex.
\end{restatable}

We remark that any $n-1$ rank-one matrices $u_1v_1^\intercal,\dots,u_{n-1}v_{n-1}^\intercal$ can be made strictly upper triangular \changetwo{simultaneously} by left- and right-multiplying by $\SO(n)$ matrices using Gram--Schmidt. In particular, optimization problems or feasibility problems with constraints on $\ip{u_iv_i^\intercal, X}$ are a special case of problems with SUT constraints (see \cref{prop:rank_one}).
This holds too for any $n-1$ coordinate constraints. See \cref{subsec:rank_one_coordinate} for a more detailed explanation.

Optimization over $\Bop(n)$ is tractable using a linearly sized SDP by \Cref{thm:convSOn}, so the presentation of this theorem also can be turned into an efficient algorithm for performing such optimization whenever $\cC$ is itself efficiently SDP-representable.

We will further show strong structural results about the matrices in $\SO(n)$ with fixed SUT entries. These structural results allow us to explicitly construct an optimizer of \eqref{eq:opt_over_so_with_sut} given an optimizer of \eqref{eq:opt_over_bop_with_sut} under the assumptions of \cref{thm:opt_so_with_sut}.
They will additionally allow us to extend \cref{thm:opt_so_with_sut} to an approximation result for $\SO(n)$ with $\det(A)<0$. These structural results are summarized below and proven in parts throughout \Cref{sec:utconstructions}.
\begin{theorem}
    \label{thm:ut_structure_dt}
    Let $\sigma \in\inter(\pi_\sut(\Bop(n)))\subseteq \R^{\binom{n}{2}}$ and let $V_\sigma = \{X \in \O(n) : \pi_\sut(X) = \sigma\}$.
    The following assertions hold:
    \begin{enumerate}
        \item $|V_\sigma| = 2^n$.
        \item For each $i \in [n]$, there exist functions $\alpha_i(\sigma) < \beta_i(\sigma)$, so that $X_{i,i} \in \{\alpha_i, \beta_i\}$ for each $X \in V_\sigma$. We will suppress the dependence of $\alpha_i$ and $\beta_i$ on $\sigma$ for convenience of notation.
        \item No two elements in $V_\sigma$ have the same diagonal entries.
        That is, for each $d \in \{\alpha_1, \beta_1\} \times \{\alpha_2, \beta_2\} \times\dots\times \{\alpha_n, \beta_n\}$, there is a unique $X \in V_\sigma$ so that $\diag(X) = d$.
\item For each $i\in[n]$, $\alpha_i$ and $\beta_i$ are continuous functions of $\sigma$. \changetwo{The function $\beta_i$ is convex in $\sigma$, and the function $\alpha_i$ is concave in $\sigma$.}
        \item $X \in V_\sigma$ is in $\SO(n)$ if and only if the number of $i$ so that $X_{i,i} = \alpha_i$ is even.
        \item Given $\rho \in \{-1, 1\}^n$, we can construct a matrix $X \in V_\sigma$ so that $X_{i,i} = \begin{cases} \alpha_i \text{ if }\rho_i = -1\\\beta_i \text{ if }\rho_i = 1\end{cases}$ in time $O(n^3)$.
    \end{enumerate}
\end{theorem}

\subsubsection{Obstructions to Progress}
Finally, \cref{sec:obstructions} provides constructions showing that \Cref{thm:diagconv}, \Cref{thm:twoconvex}, and \changetwo{\Cref{thm:feasibility}} above are \emph{\change{maximal}} in specific senses. We summarize the results here:
\begin{theorem}
    \label{thm:obstructions}
    The following assertions hold:
    \begin{enumerate}
        \item For any $n\geq 3$, the images of $\SO(n)$ under linear maps to $\R^2$ are ``maximally convex'' in the following sense: There exists $\pi : \R^{n\times n}\rightarrow \R^3$ so that $\pi(\SO(n))$ is nonconvex.
        \item The projection of $\SO(n)$ onto its diagonal is ``maximally convex'' in the following sense:
        For $A \in \R^{n\times n}$, let $\pi(X) = (X_{11}, X_{22}, \dots, X_{nn}, \ip{A, X})$. If $A$ is not itself diagonal, then $\pi(\SO(n))$ is not convex.
        \item For any $n\geq 3$, the projection of $\SO(n)$ onto its SUT entries is ``maximally convex'' in the following sense:
        If $\pi : \R^{n\times n} \rightarrow \R^m$ is any linear map with $\rank(\pi)> \binom{n}{2}$, then $\pi(\SO(n))$ is not convex.
        \item \label{itm:necessary} The assumption $\det(C)\geq 0$ in \cref{thm:opt_so_with_sut} is necessary in the following sense:
        There exists $\sigma \in \R^{\binom{n}{2}}$ and a diagonal matrix $C$ so that
        \begin{align*}
            \max_{X \in \SO(n)} \{\ip{C, X} : \pi_\sut(X) = \sigma\} &<
            \max_{X \in \conv(\SO(n))} \{\ip{C, X} : \pi_\sut(X) = \sigma\}\\
        \end{align*}
\end{enumerate}
\end{theorem}
This theorem is proven in parts throughout \Cref{sec:obstructions}.

\subsection{Related literature}

Hidden convexity results are scattered throughout the literature on optimization, numerical linear algebra, and matrix analysis. 
We recommend the following surveys/chapters for introductions to this subject~\cite{xia2020survey,polik2007survey,barvinok2002course}.
Along these lines, our hidden convexity results and their subsequent applications in deriving convex SDP relaxations of nonconvex problems parallel Dines' Theorem~\cite{dines1941mapping} and its application in deriving the S-lemma~\cite{fradkov1979thes}, a fundamental result in control theory and nonlinear optimization.

Our results extend existing hidden convexity results related to the (special) orthogonal group.
Some of the earliest work in this line is \cite[Theorem 8]{horn1954doubly} stating that $\diag(\SO(n)) = \PP_n$.
There are other similar results concerning the convexity of the image of $\SO(n)$ under various nonlinear maps, for example the famous Schur--Horn theorem~\cite{horn1954doubly}.
Another paper along these lines is \cite{fiedler2009suborthogonality}, which characterizes the possible projections of $\SO(n)$ onto its rectangular submatrices.
In particular, it is not hard to show using their results that the projection of $\SO(n)$ onto a $k \times \ell$ rectangular submatrix is convex if and only if $k + \ell \le n$. Our results extend \cite{fiedler2009suborthogonality} to nonrectangular coordinate patterns.
For further work in this direction, see \cite{TamSurvey,GSBook}.

Another important piece of related work is \cite{saunderson2015semidefinite}, which gives a description of $\conv(\SO(n))$ in terms of linear matrix inequalities (LMI).
This LMI description is constructed using Lie group theory applied to $\SO(n)$
and is related to the fact that the fundamental group of $\SO(n)$ is $\Z/2\Z$.
This fact will also be crucial in our proof of \cref{thm:oneconnected}. 
Inspired by techniques from \cite{saunderson2015semidefinite}, we can view our hidden convexity results as new quadratic convexity results on the sphere in the spirit of Brickman's Theorem~\cite{brickman1961field}. Recall, Brickman's Theorem states that for any $A,B\in\S^n$ and $n\geq 3$, that
\begin{align*}
    \set{\begin{pmatrix}
    x^\intercal Ax\\
    x^\intercal Bx
    \end{pmatrix}\in\R^2:\, x\in\bS^{n-1}}
\end{align*}
is convex. Here, $\bS^{n-1}$ is the sphere in $\R^n$.
We elaborate on this connection in \Cref{sec:quadratic_convexity}.

There are many other examples in which people consider optimization over the special orthogonal group.
These problems were implicitly studied for $\SO(3)$ in \cite{lee2011spacecraft}, where they consider a formulation in terms of quadratic maps of quaternions.
In another instance, \cite{brynte2022tightness} shows that certain standard semidefinite programming approaches to quadratic optimization problems applied to $\SO(n)$ do not always produce the correct result.
For this, they use the theory of nonnegative quadratic forms over real varieties developed in \cite{blekherman2016sums}.
Some recent work of \citet{gilman2022semidefinite}
considers the exactness of SDP relaxations of quadratic optimization problems with variables in the Stiefel manifold $\set{X \in \R^{n\by k}:\, X^\intercal X = I_k}$ for some $k \leq n$. Note that when $k=n$, that this set is identical to $\O(n)$. \citet{gilman2022semidefinite} show that the natural SDP relaxation is exact for such problems when the operator defining the quadratic form is close enough to being diagonalizable.

\section{Preliminaries}
\label{sec:prelim}
If $\pi : \R^n \rightarrow \R^m$ is a linear map, we denote by $\pi^* : \R^m \rightarrow \R^n$ the adjoint operator.

We will need to define a \emph{maximal torus} in $\SO(n)$.
Fix some $n$ for this section.
Let $k = \lfloor \frac{n}{2}\rfloor$ and
let $R(\theta_1,\dots,\theta_k)$ denote the matrix in $\SO(n)$ given by
\begin{align}
R(\theta_1,\dots,\theta_k) &\coloneqq \begin{pmatrix}
\begin{smallmatrix}
\cos(\theta_1) & \sin(\theta_1)\\
-\sin(\theta_1) & \cos(\theta_1)
\end{smallmatrix} & \\
& \ddots\\
& & \begin{smallmatrix}
    \cos(\theta_k) & \sin(\theta_k)\\
    -\sin(\theta_k) & \cos(\theta_k)
    \end{smallmatrix}
\end{pmatrix}
\label{eq:torus_even}
\end{align}
if $n$ is even, and 
\begin{align}
R(\theta_1,\dots,\theta_k) &\coloneqq \begin{pmatrix}
\begin{smallmatrix}
\cos(\theta_1) & \sin(\theta_1)\\
-\sin(\theta_1) & \cos(\theta_1)
\end{smallmatrix} & \\
& \ddots\\
& & \begin{smallmatrix}
    \cos(\theta_k) & \sin(\theta_k)\\
    -\sin(\theta_k) & \cos(\theta_k)
    \end{smallmatrix}\\
    && &\change{1}
\end{pmatrix}
\label{eq:torus_odd}
\end{align}
if $n$ is odd.

We define the maximal torus $\T$ to be the set of matrices of the form  $R(\theta_1, \dots, \theta_k)$ as the $\theta_i$ range over $[0,2\pi)$.
The following result is a special case of what is known as the Maximal Torus Theorem, but is a simple corollary of the real spectral theorem~\cite[Theorem
2.5.8]{horn2012matrix} in our setting.

\begin{theorem}
\label{lem:max_torus}
For any $X\in\SO(n)$, there exists $U\in\SO(n)$ so that $U^{\intercal} X U \in \T$. That is, $U^{\intercal}XU = R(\theta_1 ,\dots, \theta_k)$ for some $\theta_i \in [0, 2\pi)$.
\end{theorem}

For $A\in\R^{n\by n}$, let $\|A\|_{\tr}$ and $\|A\|_\textup{op}$ denote the \emph{trace norm} and \emph{operator norm} of $A$. These are defined as the sum of the singular values of $A$ and the maximum singular value of $A$ respectively.

Define the \emph{special trace} of $A\in\R^{n\by n}$ to be
\[
    \str(A) \coloneqq \max_{X \in \SO(n)} \ip{A, X}.
\]
This function is well-defined as $\SO(n)$ is compact.
Furthermore, $\str(\cdot)$ is convex and $1$-Lipschitz with respect to the trace norm.
This holds because $\str(\cdot)$ is defined as the pointwise maximum of linear functions which are individually $1$-Lipschitz with respect to the trace norm.
\change{Finally, $\str(\cdot)$ can be computed exactly using a single SVD:

\begin{lemma}[{\cite[Problem 65-1]{wahba1965least}}] \label{lem:strcompute} 
    Given $A \in \R^{n \times n}$ with singular values $\sigma_{1} \geq ... \geq \sigma_{n}$,  
    \[ \str(A) \coloneqq \max_{X \in \SO(n)} \ip{A, X} = \sum_{i=1}^{n-1} \sigma_{i} + \sign(\det(A)) \sigma_{n} .  \]
\end{lemma}
}

\change{
We will also make use of the concept of \emph{separation oracles} and the ellipsoid algorithm.
If $C \subseteq \R^n$ is a compact convex set and $x \not \in C$, then there is a hyperplane that separates $x$ and $C$.
This \emph{separating hyperplane} is given by a nonzero vector $y \in \R^n$ so that $\langle y, x \rangle \geq \max \{\langle y, c \rangle : c \in C\}$.
A $\epsilon$-\emph{weak separation oracle} for $C$ is an oracle that on an input $x \in \R^n$, either correctly declares $x \in C + \B_{\infty}(0,\epsilon)$, or outputs $y \in \R^n$ so that $y$ is a separating hyperplane between $x$ and $C$.
Here, $\B_{\infty}(a,r)$ is the ball of radius $r$ in the $L_{\infty}$ norm centered at $a$. The algorithmic equivalence between weak separation oracles and approximate optimization over convex sets is outlined in \cite{grotschel2012geometric}.

The ellipsoid algorithm as described in \cite{grotschel1981ellipsoid} provides the following guarantee for optimization in $\R^2$.
\begin{theorem}
    \label{thm:ellipsoid}
    Suppose we have access to an $\epsilon$-weak separation oracle for closed compact $C \subseteq \R^2$, we are given a $R \in \R$ so that $C \subseteq \B_2(0,R)$ and $C$ includes a ball of radius at least $\epsilon$.
    There is an algorithm that optimizes a linear function with unit $L_2$ norm over $C$ within an additive error of $\epsilon$ using at most $O(\log(\frac{R}{\epsilon}))$ calls to the weak separation oracle.
\end{theorem}
} 

\section{Feasibility problems on $\SO(n)$ with diagonal constraints}
\label{sec:feas_diag}
This section considers the feasibility problem
\begin{align}
    \label{eq:feasibility_diagonal}
    \set{X\in\SO(n):\, \diag(X)\in\cC},
\end{align}
where $\cC\subseteq\R^n$ is convex. We will assume that $\cC$ has an efficient separation oracle.

\citet[Theorem 8]{horn1954doubly} shows that $\diag(\SO(n))=\PP_n$.
As an immediate corollary, \eqref{eq:feasibility_diagonal} is feasible if and only if
$\PP_n\cap\cC$ is nonempty.

\cref{app:separation_pp} shows how to efficiently separate from $\PP_n$. 
Combined with a separation oracle for $\cC$, we may then run an ellipsoid-style algorithm for deciding feasibility of $\PP_n\cap\cC$ (up to the usual errors).
Supposing that $d\in\PP_n\cap\cC$ is found, it remains to see how to construct a witness $X\in\SO(n)$ with $\diag(X) = d$.

We will need the following description of $\PP_n$ given in \cite{Lancia2018,jeroslow1975defining}:
\begin{equation} \label{eq:parityInequalities}
    \PP_{n} = \set{ x \in [-1,1]^{n} :\, \ip{x, 1_{n} - 2 \cdot 1_{S}} \leq n-2 ,\,\quad \forall \text{ odd  } S \subseteq [n]}.
\end{equation}
Here, $1_n$ is the all-ones vector, $1_S$ is the indicator vector of the set $S$ and $S$ is odd if $\abs{S}$ is odd.

We will also need \change{a} constructive version of the Schur--Horn theorem.
\changetwo{To state this, we need the notion of \emph{ majorization}. If $c, d \in \R^n$, let $c^{\uparrow}$ and $d^{\uparrow}$ be the results of sorting the entries of $c$ and $d$ in nondecreasing order.
We say $c$ majorizes $d$ if $\sum_{i=1}^n c_i = \sum_{i=1}^n d_i$ and for each $\ell \in [n]$,
\[
    \sum_{i = 1}^{\ell} c^{\uparrow}_{i} \leq 
    \sum_{i = 1}^{\ell} d^{\uparrow}_{i}.
\]
The following lemma} is essentially due to \cite{chan1983diagonal}.
\begin{lemma}
    \label{lem:chan_li}
Given $c,\,d\in\R^n$ such that $c$ majorizes $d$, it is possible to construct a sequence of matrices
\begin{align*}
    Q_1,\dots, Q_{n-1}\in\SO(n)
\end{align*}
in time $O(n\log n)$ satisfying
\begin{align*}
    \diag\left(\left(\prod_i^{n-1} Q_i\right)^\intercal \Diag(c)\left(\prod_i^{n-1} Q_i\right)\right) = d.
\end{align*}
Furthermore, each $Q_i$ differs from the identity on only one principal $2 \times 2$ block, where it is a rotation matrix.
\end{lemma}

We are now ready to prove the following theorem.
\thmdiagonal*
\begin{proof}

    We focus on the case of even $n$ for simplicity. The odd case follows analogously.
Let $m\coloneqq n/2$ and let $\theta_{1}, ..., \theta_{m} \in [0,2\pi)$ to be fixed later. Recall the definition of $R(\theta_1 ,\dots, \theta_k)$ from \eqref{eq:torus_even} and \eqref{eq:torus_odd}, and let $c\coloneqq\diag(R(\theta_1,\dots,\theta_m))$. If we can find $\theta_1,\dots,\theta_m$ so that $c$ majorizes $d$, then we can apply \cref{lem:chan_li} to produce the required element of $\SO(n)$ with diagonal $d$. Equivalently, we may pick $c_1=c_2,c_3=c_4,\dots$ arbitrarily in $[-1,1]$ and define $\theta_i = \arccos(c_{2i})$.
    
    We will set $c_i$ as follows: Let $t \coloneqq \frac{1}{4} (n - \langle d, 1_{n} \rangle)$ and let $j-1 = \lfloor t \rfloor$ be the integer part and $\delta := t - \floor{t}$ be the fractional part of $t$. We set 
\[ c_{1} = ... = c_{2(j-1)} = -1, \qquad c_{2j+1} = ... = c_{n} = 1 ,     \]
    and the remaining elements we set as $c_{2j-1} = c_{2j} = 1-2\delta$. Note that $1-2\delta\in[-1,1]$ since the fractional part $\delta \in [0,1]$. Then, we have
    \begin{align*} 
    \langle c, 1_{n} \rangle & = - 2 (j-1) + 2 (1-2\delta) + 2 (m-j) 
    \\ & = \frac{-2}{4} (n - \langle d, 1_{n} \rangle - 4 \delta ) + 2 (1-2\delta) + \frac{2}{4} (n + \langle d, 1_{n} \rangle - 4 (1-\delta)) 
    = \langle d, 1_{n} \rangle  ,  
    \end{align*}
    where the second step was by our definition of $c$, and the third was by our choice of $j$.

    Now we verify the majorization inequalities:
    \begin{align*}
        \forall k \leq 2(j-1)&\qquad \sum_{i=1}^{k} c_{i} = -k \leq \sum_{i=1}^{k} \changetwo{d_{i}^\uparrow},\qquad\text{and}\\
        \forall k \geq 2j + 1 &\qquad \sum_{i=k}^{n} c_{i} = (n-k+1) \geq \sum_{i=k}^{n} \changetwo{d_{i}^\uparrow}. 
    \end{align*}
    Here, the last step in both inequalities hold because $d \in [-1,1]^{n}$. Since $\langle c, 1_{n} \rangle = \langle d, 1_{n} \rangle$, the second set of inequalities are equivalent to 
\[ \forall k \geq 2j\qquad \sum_{i=1}^{k} c_{i} = -k \leq \sum_{i=1}^{k} \changetwo{d_{i}^\uparrow} .    \]
    We now verify the final inequality for index $k := 2j-1$:
    \begin{align*} 
    \sum_{i=1}^{k} c_{i} & = -2(j-1) + (1-2\delta) = 1 - \frac{1}{2} (n - \langle c, 1_{n} \rangle )
    = \frac{1}{2} ( \langle d, 1_{n} \rangle - (n-2) ) \leq \sum_{i=1}^{k} \changetwo{d_{i}^\uparrow} ,  
    \end{align*}
    where the first step was by definition of $c$, in the second step we used that $j-1 = t-\delta$, in the third step we used that $\langle c, 1_{n} \rangle = \langle d, 1_{n} \rangle$, and the final step was by the defining inequalities of $\PP_{n}$ given in \cref{eq:parityInequalities} for odd set $S = [k] = [2j-1]$. 
    
    Setting $\cos(\theta_{i}) = c_{2i}$, we have $R(\theta_{1}, ..., \theta_{m}) \in \SO(n)$ with diagonal $c$ majorizing $d$.
    
    Now, apply \cref{lem:chan_li} to get a matrix $U = \prod_{i=1}^{n-1} Q_i \in\SO(n)$ such that
    \begin{align*}
        \diag(U^\intercal \Diag(c) U) = d. 
    \end{align*}
    For notational convenience, write $R$ for $R(\theta_1,\dots,\theta_m)$. Then, $U^\intercal R U\in\SO(n)$ satisfies
    \begin{align*}
        \diag(U^\intercal R U) &= \diag\left(U^\intercal \Diag(c) U\right) + \diag\left(U^\intercal (R-\Diag(c))U\right)\\
        &= \diag(U^\intercal \Diag(c)U) = d.
    \end{align*}
    Here, the second line follows as $R - \Diag(c)$ is skew symmetric so that $U^\intercal(R - \Diag(c))U$ must also be skew symmetric. In particular, $\diag(U^\intercal(R-\Diag(c))U) = 0$.

    The time complexity follows from the fact that each $Q_i$ differs from the identity only in a principal $2 \times 2$ block, so all $n-1$ conjugations by $Q_1, ..., Q_{n-1}$ can be completed in $O(n^2)$ time. \ije{\qedhere}
\end{proof}

\section{Optimization on $\SO(n)$ subject to one constraint}
\label{sec:son_one_constraint}

This section will discuss optimization of a linear function over $\SO(n)$ subject to a single (possibly two-sided) linear constraint:
\begin{align}
    \label{eq:one_constraint_problem}
    &\max_{X\in\SO(n)}\set{\ip{A,X}:\, \ip{B,X}\in[a,b]}.
\end{align}
We will provide a proof that for problems of the form \eqref{eq:one_constraint_problem}, the convex relaxation that replaces $\SO(n)$ with $\conv(\SO(n))$ is exact.
Moreover, we will give a practical algorithm for this problem that runs in roughly the same time as \changetwo{Wahba's problem without additional constraints}, i.e., in the time to compute an SVD.

The technical core of these results lies in the following theorem:
\thmoneconnected*
This theorem implies the following fact using the observation that a subset of $\R^2$ is convex if and only if its intersection with every one-dimensional affine subspace is connected.
\thmtwoconvex*

\change{
\begin{proof}
Let $x,y\in\pi(\SO(n))$ denote a pair of distinct points.
By definition, there exist $X,Y\in\SO(n)$ so that $\pi(X) = x$ and $\pi(Y) = y$.
Let $\ell\in\R^2\setminus\set{0}$, $\alpha\in\R$ parameterize the affine line so that $\ip{\ell,x} = \ip{\ell,y} = \alpha$. Thus, $\pi(\SO(n))\cap\set{z\in\R^2:\, \ip{\ell,z} = \alpha}$ is the linear image of the connected set
\begin{align*}
    \SO(n)\cap \set{Z\in\SO(n):\,\ip{\pi^*(\ell), Z} = \alpha}.
\end{align*}
In particular, $\pi(\SO(n))\cap\set{z\in\R^2:\, \ip{\ell,z} = \alpha}$ is connected so $[x,y]$ is contained in $\pi(\SO(n))$.\ije{\qedhere}
\end{proof}
}

As we have seen, this fact implies that any optimization problem of the form \eqref{eq:one_constraint_problem} can be solved as a convex optimization problem on $\conv(\SO(n))$. Thus, one could theoretically solve this problem with an exponentially sized SDP using \cref{thm:convSOn}.
Alternatively, we give an algorithm which can successfully solve any such optimization problem in $O(n^3\log^2(n))$ time.

In the case of $\SO(3)$, \cref{thm:twoconvex} can be viewed as a corollary of Brickman's theorem \cite{brickman1961field}, which states that the image of the unit sphere under a homogeneous quadratic map into $\R^2$ is always convex.
This, together with the fact that $\SO(3)$ is the image of a sphere under a quadratic map shows the result in that case (see also \cref{sec:quadratic_convexity}).
On the other hand, \cref{thm:twoconvex} does not follow directly from known quadratic convexity theorems for $n\geq 4$.

\subsection{Topological preliminaries for the proof of \Cref{thm:oneconnected}}
The proof of \Cref{thm:oneconnected} will require some topological techniques, which we review here.
We attempt to be as explicit as possible in our constructions and proofs to make them accessible to readers that are less familiar with such arguments.
As a general reference for (algebraic) topology, we refer to~\cite{MR1867354}.

As motivation, consider the (easy) problem of showing that \change{the image of $\SO(n)$ under a continuous map into $\R^1$ is connected (and hence convex).
We can see this by appealing to two facts. First,}
$\SO(n)$ is a connected set.
\change{Second,}
the image of a connected set under any continuous map is connected.
\change{
Our proof of \cref{thm:oneconnected} will be spiritually similar and will require generalizations of these two key ingredients.    
First, we will need to understand in what ways we can ``move around'' $\SO(n)$. This will be answered in \cref{lem:curve}, where we will build ``loops'' in $\SO(n)$ with nice properties. Second, we will need to understand how these loops behave under a continuous map into $\R^2$. This will be answered in \cref{lem:existencepoint}.
}

In order to formalize these statements,
\change{we will need some basic definitions from topology/homotopy theory that we present below.}
We encourage the reader to keep the following topological spaces in mind:
\begin{itemize}
    \item $\SO(n)\subseteq\R^{n\by n}$ viewed as a subtopological space of $\R^{n\by n}$ with the standard topology.
    \item The punctured plane $\R^2\setminus (0,0)$ viewed as a subtopological space of $\R^2$ with the standard topology.
\end{itemize}

Let $X$ be a topological space with a designated base point $x\in X$.
The fundamental group of $X$, denoted $\pi_1(X)$, is a group whose elements are (equivalence classes of) \change{continuous} functions $\gamma : [0,1] \rightarrow X$ so that $\gamma(0) = \gamma(1) = x$.
We will refer to such functions as \emph{loops}.
We will say that two loops $\gamma_1$ and $\gamma_2$ are equivalent if there
exists a continuous function $T : [0,1] \times [0,1] \rightarrow X$ so that $T(0,t) = \gamma_1(t)$ and $T(1,t) = \gamma_2(t)$ for all $t \in [0,1]$.
We refer to such a $T$ as a \emph{homotopy}.
Intuitively, two loops $\gamma_1$ and $\gamma_2$ are equivalent if $\gamma_1$ can be continuously deformed into $\gamma_2$.
This set of loops can be made into a group with the group operations being the concatenation of loops.

The identity element of $X$ is represented by the constant loop given by $i(t) = x$ for all $t \in [0,1]$.
If $X$ is path connected, then the fundamental group is independent of the choice of basepoint---this will be the case for all topological spaces we consider.

For example, the fundamental group of the punctured plane $\R^2\setminus(0,0)$ is $\Z$. Explicitly, any loop in $\R^2\setminus(0,0)$ can be identified with the number of times it winds anticlockwise around the origin.
Less intuitively, we will also need the fact that for $n\geq 3$, the fundamental group of $\SO(n)$ is $\Z/2\Z$.
\change{An important consequence of this fact is that given any loop in $\SO(n)$ (for $n\geq 3$), simply concatenating this loop with itself results in the identity element in $\pi_1(\SO(n))$, which can then be continuously deformed to a point.}
This will be used in \cref{lem:existencepoint} below, for which
\cref{fig:proof_lem3} shows a cartoon of the proof strategy.

This will be relevant to us in the following context:
\begin{lemma}
    \label{lem:existencepoint}
    Let $n  \ge 3$.
    Suppose $f : \SO(n) \rightarrow \R^2$ is a continuous map and $\gamma:[0,1]\to \SO(n)$ is a loop such that $(0,0)$ is not in the image $(f_*(\gamma))([0,1])$. In this case, we may view $f_*(\gamma)$ as a loop in $\R^2\setminus(0,0)$.
    If $f_*(\gamma)$ is not equivalent to the identity element in $\pi_1(\R^2\setminus(0,0))$, then $(0,0)\in f(\SO(n))$.
\end{lemma}
\begin{proof}
\change{
We may assume that $\gamma$ is equivalent to the identity element in $\pi_1(\SO(n))$ by concatenating $\gamma$ with itself. Indeed, this produces the identity element in $\pi_1(\SO(n))$ as  $\pi_1(\SO(n))=\Z/2\Z$.
Additionally, if $f_*(\gamma)$ is not equivalent to the identity element in $\pi_1(\R^2\setminus(0,0))$, then concatenating this loop with itself does not yield the identity element.

Let $T$ be a homotopy between $\gamma$ and the identity element $i$ in $\pi_1(\SO(n))$, i.e., assume $T(0,t) = \gamma(t)$ and $T(1,t) = i(t)$.
Now, consider the image $(f\circ T) ([0,1]\times[0,1])$. We claim that $(0,0)$ is in this image. Indeed, supposing otherwise, 
we have constructed a homotopy between $f_*(\gamma)$ and the identity element in $\pi_1(\R^2\setminus (0,0))$, a contradiction.\ije{\qedhere}
}
\end{proof}

By adding a translation term to $f$, we can apply \cref{lem:existencepoint} with an arbitrary point $\beta\in\R^2$ in place of the origin.
\Cref{fig:projectiontorus}
depicts a two-dimensional linear image of $\SO(3)$ and a loop $\gamma$ in $\SO(3)$ and gives some intuition on how we will use \cref{lem:existencepoint} to prove \cref{thm:oneconnected}.

\begin{figure}
    \centering
    \def\svgscale{0.75}
    \begingroup \makeatletter \providecommand\color[2][]{\errmessage{(Inkscape) Color is used for the text in Inkscape, but the package 'color.sty' is not loaded}\renewcommand\color[2][]{}}\providecommand\transparent[1]{\errmessage{(Inkscape) Transparency is used (non-zero) for the text in Inkscape, but the package 'transparent.sty' is not loaded}\renewcommand\transparent[1]{}}\providecommand\rotatebox[2]{#2}\newcommand*\fsize{\dimexpr\f@size pt\relax}\newcommand*\lineheight[1]{\fontsize{\fsize}{#1\fsize}\selectfont}\ifx\svgwidth\undefined \setlength{\unitlength}{387.10231991bp}\ifx\svgscale\undefined \relax \else \setlength{\unitlength}{\unitlength * \real{\svgscale}}\fi \else \setlength{\unitlength}{\svgwidth}\fi \global\let\svgwidth\undefined \global\let\svgscale\undefined \makeatother \begin{picture}(1,0.34360314)\lineheight{1}\setlength\tabcolsep{0pt}\put(0,0){\includegraphics[width=\unitlength,page=1]{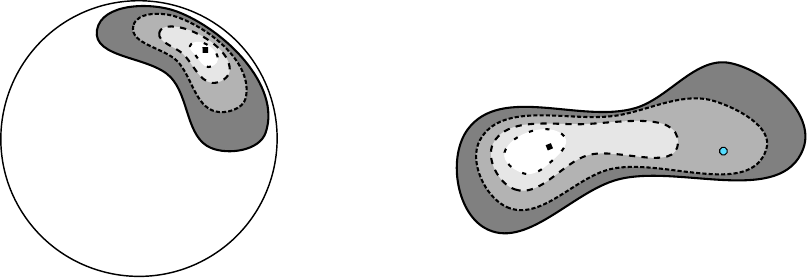}}
\put(0.43105375,0.14578714){\color[rgb]{0,0,0}\makebox(0,0)[lt]{\lineheight{1.25}\smash{\begin{tabular}[t]{l}$\stackrel{f}{\longrightarrow}$\end{tabular}}}}\put(0.06739619,0.10043035){\color[rgb]{0,0,0}\makebox(0,0)[lt]{\lineheight{1.25}\smash{\begin{tabular}[t]{l}$\SO(n)$\end{tabular}}}}\put(0.80328983,0.28258226){\color[rgb]{0,0,0}\makebox(0,0)[lt]{\lineheight{1.25}\smash{\begin{tabular}[t]{l}$\R^2$\end{tabular}}}}\end{picture}\endgroup      \caption{A cartoon of the proof of \cref{lem:existencepoint}. \change{On the left,} we depict a loop (solid line) in $\SO(n)$, that is equivalent to the identity element, and its deformations (dashed lines) to a point.
    \change{On the right is the \emph{image} of the loop and its deformations under the map $f:\SO(n)\to\R^2$.
    The origin in $\R^2$ is shown in blue.
    By assumption, $f_*(\gamma)$ ``winds'' around the origin a nonzero number of times.
    We will show that the image of the deformations must eventually ``pass through'' the origin. This shows that the origin is in the image of $f(\SO(n))$.
    }}
    \label{fig:proof_lem3}
\end{figure}

\begin{figure}
    \centering
    \includegraphics*[width=0.5\linewidth]{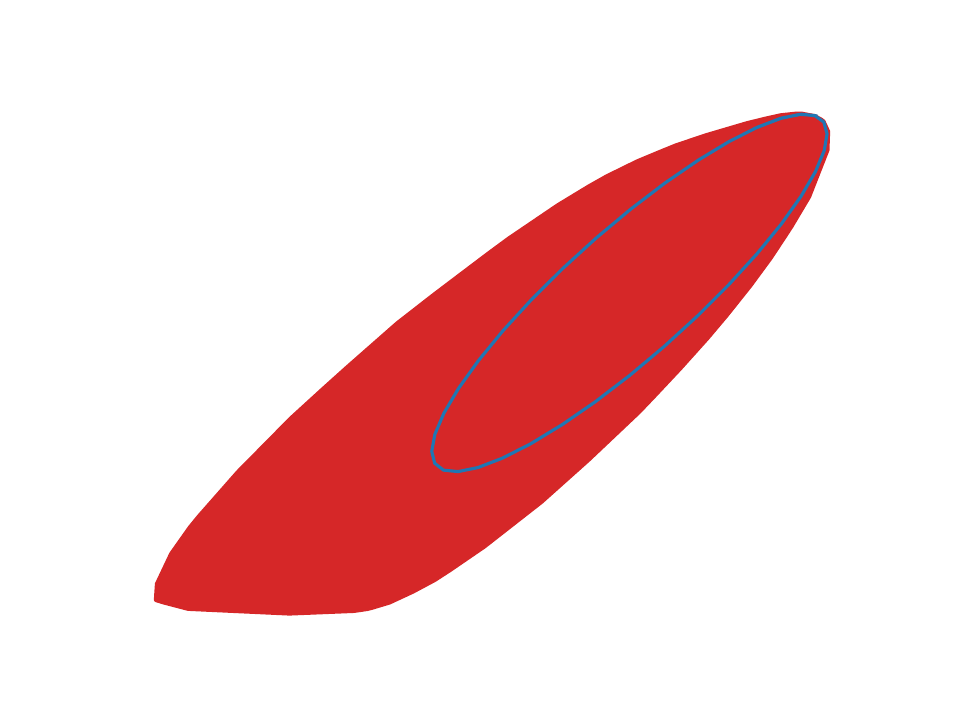}
    \caption{The solid red area is the image of $\SO(3)$ under a linear map into $\R^2$. The blue ellipse is the image of a loop in $\SO(3)$ under the same linear map.
    The (image of the) loop begins at some point on the blue ellipse and walks clockwise around the ellipse exactly once. 
    \cref{lem:existencepoint} implies that every point in $\R^2$ contained within this blue ellipse must be contained in the image of $\SO(3)$.
    The proof of \cref{thm:oneconnected} will follow a similar idea with a continuous but nonlinear function. }
    \label{fig:projectiontorus}
\end{figure}

\subsection{Proof of \Cref{thm:oneconnected}}
This subsection will contain a proof of \Cref{thm:oneconnected}.
Let $H = \set{X\in\R^{n\by n}:\, \ip{A,X} = c}$.
We will require the following \change{construction}.
\begin{lemma}
    \label{lem:curve}
Let $U,V\in H\change{\cap \SO(n)}$. Then, there exists a continuous function $\gamma : [0,1] \rightarrow \SO(n)$ with the following properties:
    \begin{itemize}
        \item $\gamma(0) = \gamma(1) = U$,
        \item $\gamma(\frac{1}{2}) = V$,
        \item either $\ip{A,\gamma(t)} = c$ for all $t \in (0,\frac{1}{2})$ or $\ip{A, \gamma(t)} > c$ for all $t\in(0,\frac{1}{2})$, and
        \item either $\ip{A,\gamma(t)} = c$ for all $t \in (\frac{1}{2},1)$ or $\ip{A, \gamma(t)} < c$ for all $t\in(\frac{1}{2},1)$.
    \end{itemize}
\end{lemma}

\begin{proof}

    \change{We may write $U = R(\eta_1,\dots,\eta_k)$ for some $\eta_i\in[0,2\pi)$} and $V = R(\theta_1,\dots,\theta_k)$ for some $\theta_i\in[0,2\pi)$.
    
    Let $\bS^1$ denote the unit circle in $\R^2$, thought of as the points $[0,2\pi)$ where $0$ and $2\pi$ are identified.
    Let $\T$ denote the set of all matrices $R(\phi_1, \dots, \phi_k)$ as $\phi_1,\dots,\phi_k$ range over $\bS^1$.

    \change{
        Now, consider the expression $\ip{A, R(\phi_1,\dots,\phi_k)}$, which we can express as a sum over the blocks of $R(\phi_1, \dots, \phi_k)$:
        \[
            \ip{A, R(\phi_1,\dots,\phi_k)} = 
            \sum_{\ell = 1}^{k}
            \ip{\begin{pmatrix} A_{2i - 1, 2i - 1} & A_{2i - 1, 2i}\\  A_{2i, 2i - 1} & A_{2i, 2i}\end{pmatrix}, \begin{pmatrix} \cos(\phi_i)& \sin(\phi_i)\\-\sin(\phi_i)&\cos(\phi_i) \end{pmatrix}} + 
            c_0,
        \]
        where $c_0 = A_{nn}$ if $n$ is odd and 0 otherwise.
        We may rewrite this as
        \begin{align*}
            \ip{A, R(\phi_1,\dots,\phi_k)} &= 
            \sum_{\ell = 1}^{k}
            \ip{\begin{pmatrix} A_{2i - 1, 2i - 1} +  A_{2i, 2i} \\ A_{2i - 1, 2i} - A_{2i, 2i - 1}\end{pmatrix}, \begin{pmatrix} \cos(\phi_i)\\\sin(\phi_i)\end{pmatrix}} +
            c_0\\
            &= 
            \sum_{i=1}^k c_i \ip{\begin{pmatrix}
            \cos(\hat\phi_i)\\
            \sin(\hat\phi_i)
            \end{pmatrix}, \begin{pmatrix}
            \cos(\phi_i)\\
            \sin(\phi_i)
            \end{pmatrix}}
            + c_0\\
            &= \sum_{i=1}^k c_i \cos(\hat\phi_i - \phi_i)
                + c_0,
        \end{align*}
        where $c_i \ge 0$ and $\hat{\phi}_i \in \bS^1$ are the polar coordinates for $\begin{pmatrix} A_{2i - 1, 2i - 1} +  A_{2i, 2i} \\ A_{2i - 1, 2i} - A_{2i, 2i - 1}\end{pmatrix}$.
    }

        \change{
        We will construct the loop $\gamma(t)$ by setting $\gamma(t) = R(\phi_1(t), \dots, \phi_k(t))$ where $\phi_i(t)$ will be defined shortly. Specifically, $\phi_i(t)$ will be a continuous function satisfying $\phi_i(0) = \phi_i(1) =\eta_i$ and $\phi_i(\frac{1}{2})=\theta_i$. In particular, $\gamma(t)$ will be continuous and satisfy $\gamma(0)=\gamma(1) = U$ and $\gamma(1/2) = V$.

        Now, fix some $i\in[k]$.
        We will parameterize $\phi_i(t) = \hat\phi_i - \delta(t)$. Note the requirement $\phi_i(0)=\phi_i(1) = \eta_i$ and $\phi_i(1/2) = \theta_i$ is equivalent to $\delta(0) = \delta(1) = \hat\phi_i - \eta_i$ and $\delta(1/2) = \hat\phi_i - \theta_i$. Note that $\cos(\hat\phi_i - \phi_i(t)) = \cos(\delta(t))$ is simply the ``$x$-coordinate'' of $\delta(t)$.

        Consider the locations of $\hat\phi_i - \eta_i$, $\hat\phi_i - \theta_i$ on the unit circle $\bS^1$.
        If $\hat\phi_i - \eta_i$ is on the bottom half of the unit circle, then let $\delta(3/4) = \pi$ and set $\delta(t)$ to move counterclockwise along the bottom half of the circle until $\delta(1) = \hat\phi_i - \eta_i$. Also set $\delta(0) = \hat\phi_i - \eta_i$ and set $\delta(t)$ to move counterclockwise along the bottom half of the circle until $\delta(1/4) = 0$.
        On the other hand, if $\hat\phi_i - \eta_i$ is on the top half of the unit circle, then let $\delta(3/4) = \pi$ and set $\delta(t)$ to move clockwise along the top half of the circle until $\delta(1) = \hat\phi_i - \eta_i$. Also set $\delta(0) = \hat\phi_i - \eta_i$ and set $\delta(t)$ to move clockwise along the top half of the circle until $\delta(1/4) = 0$.
        Note that in both cases, the function $\cos(\delta(t))$ is either constant or strictly increasing on each interval $[0,1/4]$ and $[3/4,1]$.
        Completely analogous ideas can be used to define $\delta(t)$ on the intervals $[1/4,1/2]$ and $[1/2,3/4]$ so that $\delta(1/2) = \hat\phi_i - \theta_i$ and $\cos(\delta(t))$ is either constant or strictly decreasing on each interval $[1/4,1/2]$ and $[1/2,3/4]$.
        (See \cref{fig:proof_one_connected} for two examples of this construction.)

        }

        We then define $\gamma(t) \coloneqq R(\phi_1(t), \phi_2(t), \dots, \phi_k(t))$.
        We see that this loop satisfies all of the properties desired based on properties of the individual $\phi_i(t)$.\ije{\qedhere}
\end{proof}
    
    \begin{figure}
        \centering
        \begin{tikzpicture}[scale=0.8]
            \begin{scope}[shift={(-5,-5)}]
                \begin{scope}[rotate=135]
                    \draw[ultra thick, ->] (1,0) arc (0:-135:1);
                \end{scope}
                \draw (0,0) circle (1);
                \draw[->,rotate=0] (0,0) -- (1.5,0) node[anchor=west] {$0$};
                \draw[->,rotate=90] (0,0) -- (1.5,0) node[anchor=west] {$\hat\phi_i - \theta_i$};
                \draw[->,rotate=135] (0,0) -- (1.5,0) node[anchor=east] {$\hat\phi_i - \eta_i$};
    
                \node at (0,-2) {$t=\left[0,\tfrac{1}{4}\right]$};
                \end{scope}
        
            \begin{scope}[shift={(0,-5)}]
                    \draw[ultra thick, ->] (1,0) arc (0:90:1);
                \draw (0,0) circle (1);
                \draw[->,rotate=0] (0,0) -- (1.5,0) node[anchor=west] {$0$};
                \draw[->,rotate=90] (0,0) -- (1.5,0) node[anchor=west] {$\hat\phi_i - \theta_i$};
                \draw[->,rotate=135] (0,0) -- (1.5,0) node[anchor=east] {$\hat\phi_i - \eta_i$};
                \node at (0,-2) {$t=\left[\tfrac{1}{4},\tfrac{1}{2}\right]$};
                \end{scope}

            \begin{scope}[shift={(5,-5)}]
                \begin{scope}[rotate=90]
                    \draw[ultra thick, ->] (1,0) arc (0:90:1);
                \end{scope}
                \draw (0,0) circle (1);
                \draw[->,rotate=0] (0,0) -- (1.5,0) node[anchor=west] {$0$};
                \draw[->,rotate=90] (0,0) -- (1.5,0) node[anchor=west] {$\hat\phi_i - \theta_i$};
                \draw[->,rotate=135] (0,0) -- (1.5,0) node[anchor=east] {$\hat\phi_i - \eta_i$};
                \node at (0,-2) {$t=\left[\tfrac{1}{2},\tfrac{3}{4}\right]$};
                \end{scope}
    
            \begin{scope}[shift={(10,-5)}]  
                \begin{scope}[rotate=180]
                    \draw[ultra thick, ->] (1,0) arc (0:-45:1);
                \end{scope}
                \draw (0,0) circle (1);
                \draw[->,rotate=0] (0,0) -- (1.5,0) node[anchor=west] {$0$};
                \draw[->,rotate=90] (0,0) -- (1.5,0) node[anchor=west] {$\hat\phi_i - \theta_i$};
                \draw[->,rotate=135] (0,0) -- (1.5,0) node[anchor=east] {$\hat\phi_i - \eta_i$};
                \node at (0,-2) {$t=\left[\tfrac{3}{4},1\right]$};
                \end{scope}

                \begin{scope}[shift={(-5,0)}]
                    \begin{scope}[rotate=90]
                        \draw[ultra thick, ->] (1,0) arc (0:-90:1);
                    \end{scope}
                    \draw (0,0) circle (1);
                    \draw[->,rotate=0] (0,0) -- (1.5,0) node[anchor=west] {$0$};
                    \draw[->,rotate=-135] (0,0) -- (1.5,0) node[anchor=east] {$\hat\phi_i - \theta_i$};
                    \draw[->,rotate=90] (0,0) -- (1.5,0) node[anchor=east] {$\hat\phi_i - \eta_i$};
                    \node at (0,-2) {$t=\left[0,\tfrac{1}{4}\right]$};
                    \end{scope}
            
                \begin{scope}[shift={(0,0)}]
                        \draw[ultra thick, ->] (1,0) arc (0:-135:1);
                    \draw (0,0) circle (1);
                    \draw[->,rotate=0] (0,0) -- (1.5,0) node[anchor=west] {$0$};
                    \draw[->,rotate=-135] (0,0) -- (1.5,0) node[anchor=east] {$\hat\phi_i - \theta_i$};
                    \draw[->,rotate=90] (0,0) -- (1.5,0) node[anchor=east] {$\hat\phi_i - \eta_i$};
                    \node at (0,-2) {$t=\left[\tfrac{1}{4},\tfrac{1}{2}\right]$};
                    \end{scope}

                \begin{scope}[shift={(5,0)}]
                    \begin{scope}[rotate=-135]
                        \draw[ultra thick, ->] (1,0) arc (0:-45:1);
                    \end{scope}
                    \draw (0,0) circle (1);
                    \draw[->,rotate=0] (0,0) -- (1.5,0) node[anchor=west] {$0$};
                    \draw[->,rotate=-135] (0,0) -- (1.5,0) node[anchor=east] {$\hat\phi_i - \theta_i$};
                    \draw[->,rotate=90] (0,0) -- (1.5,0) node[anchor=east] {$\hat\phi_i - \eta_i$};
                    \node at (0,-2) {$t=\left[\tfrac{1}{2},\tfrac{3}{4}\right]$};
                    \end{scope}
        
                \begin{scope}[shift={(10,0)}]  
                    \begin{scope}[rotate=180]
                        \draw[ultra thick, ->] (1,0) arc (0:-90:1);
                    \end{scope}
                    \draw (0,0) circle (1);
                    \draw[->,rotate=0] (0,0) -- (1.5,0) node[anchor=west] {$0$};
                    \draw[->,rotate=-135] (0,0) -- (1.5,0) node[anchor=east] {$\hat\phi_i - \theta_i$};
                    \draw[->,rotate=90] (0,0) -- (1.5,0) node[anchor=east] {$\hat\phi_i - \eta_i$};
                    \node at (0,-2) {$t=\left[\tfrac{3}{4},1\right]$};
                    \end{scope}
        \end{tikzpicture}
        \caption{\change{Two examples of the construction of $\delta(t)$ in the proof of \cref{lem:curve}. 
        The first row corresponds to $\hat\phi_i - \eta_i = \tfrac{\pi}{4}$ and $\hat\phi_i - \theta_i = \tfrac{-3\pi}{4}$.
        The second row corresponds to $\hat\phi_i -\eta_i= \tfrac{3\pi}{4}$ and $\hat\phi_i - \theta_i = \tfrac{\pi}{2}$.}}
        \label{fig:proof_one_connected}
    \end{figure}

We are now ready to prove \cref{thm:oneconnected}.

\begin{proof}[Proof of \cref{thm:oneconnected}]

Suppose for the sake of contradiction that $H \cap \SO(n)$ is not connected, which by definition means that there exist nonempty closed sets $\cU, \cV \subseteq H \cap \SO(n)$ so that $\cU \cap \cV = \varnothing$, and $\cU \cup \cV = H \cap \SO(n)$.
As $\cU$ is closed, the distance function
\begin{align*}
    \dist_\cU(X) \coloneqq \min_{U\in\cU}\norm{U - X}_\textup{op}
\end{align*}
is well-defined. Let $\delta\coloneqq \min_{V\in\cV}\dist_\cU(V)$. As $\cU$ and $\cV$ are compact and disjoint, we have that $\delta>0$.
Define $f:\SO(n)\to\R^2$ given by
\begin{align*}
    f(X) = \begin{pmatrix}
    \ip{A,X} - c\\
    \dist_\cU(X) - \delta/2
    \end{pmatrix}
\end{align*}
\change{We claim that $(0,0)\notin f(\SO(n))$. Assume for contradiction that $(0,0)\in f(\SO(n))$, i.e., there exists an $X \in \SO(n)$ so that $\langle A, X\rangle = c$ and $\dist_\cU(X) = \frac{\delta}{2}$. As $X\in H\cap\SO(n)$, we have that either $X\in U$ or $X\in V$. In the first case, $\dist_\cU(X) = 0$, a contradiction. In the second case, we have that $\min_{V \in \cV} \dist_\cU(V)\leq\dist_\cU(X) = \delta/2 < \delta = \min_{V \in \cV} \dist_\cU(V)$, again a contradiction. Thus, $(0,0)\notin f(\SO(n))$.}

\change{
Now, fix $U\in\cU$ and $V\in\cV$ and let $\gamma$
be the loop constructed in \Cref{lem:curve}. We claim that $\ip{A,\gamma(t)}>c$ for all $t\in(0,\frac{1}{2})$. Indeed, by \Cref{lem:curve}, if this is not the case, then $\ip{A,\gamma(t)} = c$ for all $t\in[0,\frac{1}{2}]$ so that the path $\gamma(t)$ for $t\in[0,1/2]$ connects $U,V$ in $H\cap\SO(n)$. This contradicts the fact that $U$ and $V$ are not connected in $H\cap\SO(n)$.

}

We now verify that $f$ and $\gamma$ satisfy the assumptions of \cref{lem:existencepoint}.
To do so, we will exhibit a homotopy from $f_*(\gamma)$ to the loop $(\sin(2\pi t),-\cos(2\pi t))$.
Let
\begin{align*}
    T(s,t) \coloneqq s f_*(\gamma)(t)+(1-s)(\sin(2 \pi t), -\cos(2 \pi t)).
\end{align*}
This is clearly a continuous function from $[0,1]\times[0,1]\to\R^2$. Thus, to verify that it is a valid homotopy in $\R^2\setminus(0,0)$ it remains to check that $T(s,t)\neq(0,0)$ for any $(s,t)\in[0,1]\times[0,1]$.
To see this, note that for $t \in (0,\frac{1}{2})$, we have $\ip{A,\gamma(t)} > c$ so that $f(\gamma(t))_1 > 0$.
Additionally, for all $t \in (0,\frac{1}{2})$, $\sin(2\pi t) > 0$.
Thus, $T(s,t)_1 \neq 0$ for all $t\in(0,\frac{1}{2})$ and $s\in[0,1]$.
Similar arguments show that $T(s,t)_1\neq 0$ for all $t\in(\frac{1}{2},1)$ and $s\in[0,1]$, and that
$T(s,t)_2\neq 0$ for all $t \in\set{0,\frac{1}{2},1}$ and $s\in[0,1]$.

Finally, \cref{lem:existencepoint} implies that $(0,0)\in f(\SO(n))$, a contradiction. We conclude that $H\cap \SO(n)$ is connected.\ije{\qedhere}
\end{proof}

\subsection{Algorithms for \changetwo{One} Constraint Optimization}
Here, we aim to solve the optimization problem given in \eqref{eq:ndmax}.
\thmtwodalgo*

\change{
To state the algorithm we note that for any $A, B \in \R^{n\times n}$, \eqref{eq:ndmax} is equivalent to the following optimization problem
\begin{equation*}
    \max_{x \in \R^{2}} \{x_1 : x \in \pi(\SO(n)), x_{2} \in [a,b] \}.
\end{equation*}
Here, $\pi(\SO(n))$ is the image of $\SO(n)$ in $\R^2$ given by $\pi(X) = (\langle A, X\rangle, \langle B, X\rangle)$.
By \Cref{thm:twoconvex}, we have that $\pi(\SO(n))$ is convex. Therefore the above is a linear optimization problem over convex set $C \coloneqq \pi(\SO(n)) \cap \{x_{2} \in [a,b]\}$, and we can apply standard methods from convex optimization to \change{solve the above equivalent formulation of \eqref{eq:ndmax}}.
}

For this, we appeal to the ellipsoid algorithm, \change{which we stated in \Cref{thm:ellipsoid}}.
\change{
By \Cref{thm:ellipsoid}, it suffices to provide a weak separation oracle for the set $C = \pi(\SO(n)) \cap \{x_{2} \in [a,b]\}$. As the second constraint is trivial to separate over, we focus on separating over the convex set $\pi(\SO(n))$.

It will be useful to have a subroutine for minimizing $h(y)$ over the unit $\ell_1$-ball in $\R_2$, where
\[
    h(y) \coloneqq \str(\pi^*(y)) - \langle y, x\rangle.
\]

\begin{lemma}
    \label{lem:subroutine}
    Given $x \in \R^{2}$ with $\|x\|_{\infty} \leq 1 + \epsilon$, we can construct $\widehat{y}$ with $\norm{\widehat y}_1=1$ so that
    \begin{align*}
        h(\widehat y) - \epsilon \leq \min_y \{h(y) : \|y\|_1 = 1\}
    \end{align*}
    using at most $O\left(\log\left(\frac{1}{\epsilon}\right)\right)$ evaluations of $h$ and additional computations.
\end{lemma}
\begin{proof}

    We note that $\{y : \|y\|_1=1\}$ is a union of 4 line segments, so minimizing $h$ on this set can be done by minimizing the following 4 univariate functions on $[0,1]$:
    \[
        g_{\sigma_1\sigma_2}(\alpha) = h(\sigma_1 \alpha, \sigma_2 (1-\alpha)) = \str(\sigma_1\alpha A + \sigma_2(1-\alpha)B) - \sigma_1 \alpha x_1 - \sigma_2 (1-\alpha) x_2,
    \]
    indexed by $\sigma_1,\sigma_2\in\set{\pm1}$.
    Each of the four functions $g_{\sigma_1\sigma_2}$ is a one-dimensional convex function with Lipschitz constant bounded by 
    \begin{align*}
        \norm{A}_{\tr} + \norm{B}_{\tr} + \norm{x}_1 \leq 4+2\epsilon.
    \end{align*}
For each $\sigma_1\sigma_2\in\set{\pm 1}$, we may use golden section search \cite{kiefer1953sequential} to find a $\widehat\alpha_{\sigma_1\sigma_2}\in[0,1]$ such that
    \begin{align*}
    g_{\sigma_1\sigma_2}(\widehat\alpha_{\sigma_1\sigma_2}) \leq \min_{\alpha\in[0,1]}g_{\sigma_1\sigma_2}(\alpha) + \epsilon.
    \end{align*}
    Each application of golden section search requires
        $O\left(\log\left(\frac{1}{\epsilon}\right)\right)$
    evaluations of $g_{\sigma_1\sigma_2}$, or equivalently, evaluations of $h$.
\ije{\qedhere}
\end{proof}
}

\begin{lemma}
    Let $n\geq 3$ and $A,B\in\R^{n\by n}$ with $\norm{A}_{\tr} = \norm{B}_{\tr} = 1$.
    There is a weak separation oracle for the set $\pi(\SO(n))$ that runs in time $O(n^3\log(\frac{1}{\epsilon}))$.
\end{lemma}
\begin{proof}
    Suppose we are given $A, B \in \R^{n \times n}$ and $x \in \R^2$.
    If $\|x\|_\infty > 1 + \epsilon$, then in fact, $x \not \in \pi(\SO(n)) + \B_{\infty}(0, \epsilon)$ as, by Holder's inequality,
    \[ X \in \SO(n) \implies \max\{ |\langle A, X \rangle|, |\langle B, X \rangle| \} \leq \|X\|_{op} \max\{ \norm{A}_{\tr}, \norm{B}_{\tr}\} \leq 1 ,   \]
    where the last step was by our assumption $\norm{A}_{\tr} = \norm{B}_{\tr} = 1$. Therefore, in this case, we may immediately terminate with one of $(\pm1,0)$ or $(0,\pm1)$ as a separating hyperplane.
For the remainder, we assume that $\|x\|_\infty \leq 1 + \epsilon$.

    A nonzero vector $y\in\R^2$ defines a separating hyperplane between $x$ and $\change{\pi(\SO(n))}$ if and only if
    \[
        \langle y, x\rangle \geq \max_{X \in \SO(n)} \langle y, \pi(X)\rangle.
    \]
    Recalling the definition of $\str(\cdot)$ from \cref{sec:prelim}, the expression on the right can be written as
    \[
         \max_{X \in \SO(n)} \langle y, \pi(X)\rangle = \max_{X \in \SO(n)} \langle \pi^*(y), X\rangle = \str(\pi^*(y)).
    \]
    \change{
    Thus, a nonzero $y \in \R^{2}$ defines a separating hyperplane if and only if $h(y) \leq 0$.
    Note that by \Cref{lem:strcompute}, we can compute $h(y)$ for a given $y$ using a single SVD. 
    As $h$ is 1-homogeneous, such a $y$ exists if and only if one exists with $\|y\|_1 = 1$.

    Now, we apply \Cref{lem:subroutine} to compute $\widehat{y}$ approximately minimizing $h(y)$ on the unit $\ell_1$ ball.
    }
 
    If $h(\widehat{y}) \leq 0$, then we may output $\widehat y$ as a separating hyperplane.
    For the remainder of the proof, suppose $h(\widehat{y}) > 0$. By \cref{lem:subroutine} and $1$-homogeneity, $h(y) > -\epsilon$ for all $y \in \B_{1}(0,1)$.
    We claim that $x \in \pi(\SO(n)) + \B_{\infty}(0, \epsilon)$.
    If, to the contrary, $x \not \in \pi(\SO(n)) + \B_{\infty}(0, \epsilon)$, then by the separating hyperplane theorem, there would be some $y$ so that
    \[
        \langle y, x \rangle \geq \max \{\langle y, c+\delta\rangle : c \in \pi(\SO(n)), \delta \in \B_{\infty}(0, \epsilon)\} = 
        \max \{\langle y, c\rangle : c \in \pi(\SO(n))\} + \epsilon \|y\|_1.
    \]
    In particular, there would be some $y$ with $\|y\|_1 = 1$ such that
    \[
        h(y) = \str(\pi^*(y)) - \langle y, x\rangle \leq -\epsilon,
    \]
    which is a contradiction.\ije{\qedhere}
\end{proof}

\section{Optimization on $\SO(n)$ and $\O(n)$ with strict upper triangular constraints}
\label{sec:sut_constraints}

In this section, we consider optimization problems with strict upper triangular (SUT) constraints over $\SO(n)$ and $\O(n)$: Let $A\in\R^{n\by n}$ and let $\cC$ be a nonempty closed convex subset of $\pi_\sut(\Bop(n))$. We consider the problems (which we also stated in the introduction)
\begin{align}
    &\sup_{X\in\SO(n)}\set{\ip{A,X}:\, \pi_\sut(X)\in \cC} \tag{\ref{eq:opt_over_so_with_sut}}\\
    &\qquad\leq \sup_{X\in\O(n)}\set{\ip{A,X}:\, \pi_\sut(X)\in \cC}\tag{\ref{eq:opt_over_o_with_sut}}\\
    &\qquad\leq 
    \max_{X\in\Bop(n)}\set{\ip{A,X}:\, \pi_\sut(X)\in \cC}. \tag{\ref{eq:opt_over_bop_with_sut}}
\end{align}
where the last inequality is because $\Bop(n)$ is the convex hull of $\O(n)$.
Here, we define the values of \labelcref{eq:opt_over_so_with_sut,eq:opt_over_o_with_sut} to be $-\infty$ whenever they are infeasible. 
\change{Since \Cref{eq:opt_over_bop_with_sut} is always feasible (by our assumptions on $\cC$), and $\Bop(n)$ is compact, the supremum is always acheived in this last case.}
Nonetheless, by compactness, \labelcref{eq:opt_over_so_with_sut,eq:opt_over_o_with_sut} both achieve their maxima as long as they are feasible.

The following theorem is the main result of this section and shows that one or both of these inequalities hold at equality for certain choices of $A$.

\thmoptsosut*
We will prove \cref{thm:opt_so_with_sut} in \cref{subsec:proof_opt_so_with_sut}. As a byproduct of the proof, we will also see a numerical method for constructing optimizers of \eqref{eq:opt_over_so_with_sut} or \eqref{eq:opt_over_o_with_sut} from an optimizer of the convex program \eqref{eq:opt_over_bop_with_sut} by solving an SDP.
In \cref{subsec:rank_one_coordinate}, we verify that our results in this section can be applied to problems with few coordinate constraints or rank-one constraints.

Before moving on, we note two hidden convexity properties implied by \cref{thm:opt_so_with_sut}.

\thmfeasibility*
\begin{proof}
It is clear that $\pi_\sut(\SO(n))\subseteq\pi_\sut(\O(n))\subseteq\pi_\sut(\Bop(n))$. Now, let $\sigma\in\pi_\sut(\Bop(n))$ and set $A = I$ and $\cC=\set{\sigma}$. \cref{thm:opt_so_with_sut} implies the feasibility of \eqref{eq:opt_over_so_with_sut}, i.e., $\sigma\in\pi_\sut(\SO(n))$, whence $\pi_\sut(\Bop(n))\subseteq\pi_\sut(\SO(n))$.\ije{\qedhere}
\end{proof}
\begin{corollary}
\label{thm:hidden_convexity_so_with_sut}
Let $\change{A}\in\R^{n\by n}$ be a diagonal matrix, then
\begin{align*}
    \set{\begin{pmatrix}
    \ip{A,X} - \gamma\\
    \pi_\sut(X)
    \end{pmatrix}:\, \begin{array}{l}
        X\in\O(n)\\
        \gamma \geq 0
    \end{array}} = 
    \set{\begin{pmatrix}
    \ip{A,X} - \gamma\\
    \pi_\sut(X)
    \end{pmatrix}:\,
    \begin{array}{l}
        X\in\Bop(n)\\
        \gamma \geq 0
    \end{array}}.
\end{align*}
If additionally $\det(A)\geq 0$, then
\begin{align*}
    \set{\begin{pmatrix}
        \ip{A,X} - \gamma\\
        \pi_\sut(X)
        \end{pmatrix}:\, \begin{array}{l}
            X\in\SO(n)\\
            \gamma \geq 0
        \end{array}} = 
        \set{\begin{pmatrix}
        \ip{A,X} - \gamma\\
        \pi_\sut(X)
        \end{pmatrix}:\,
        \begin{array}{l}
            X\in\Bop(n)\\
            \gamma \geq 0
        \end{array}}.
        \end{align*}
\end{corollary}
\begin{proof}
We prove only the first claim as the second is proved analogously.
For convenience, let $\cL$ and $\cR$ denote the left- and right-hand side sets in the first claim.
As $\O(n)\subseteq\Bop(n)$, we have that $\cL\subseteq\cR$. Now, suppose $(v,\sigma)\in\cR$, so there exists $X \in \Bop(n)$ such that $\pi_{\sut}(X) = \sigma, \langle A, X \rangle \geq v$. Let
\begin{align*}
    v' = \max_{X\in\Bop(n)}\set{\ip{A,X}:\, \pi_\sut(X) = \sigma}.
\end{align*}
By definition, $v'\geq v$. Next, by \cref{thm:opt_so_with_sut}, there exists $X\in\O(n)$ such that $\pi_\sut(X) = \sigma$ and $\ip{A,X} = v'$. Then $(v',\sigma)\in\cL$.
\change{As $\cL$ is closed under decreasing its first coordinate}, $(v,\sigma)\in\cL$. As $(v,\sigma)\in\cR$ was arbitrary, we conclude $\cR\subseteq\cL$.\ije{\qedhere}
\end{proof}

\subsection{Proof of \cref{thm:opt_so_with_sut}}
\label{subsec:proof_opt_so_with_sut}

We begin by proving the following special case of \cref{thm:opt_so_with_sut}.
\begin{proposition}
\label{prop:opt_so_with_sut_generic}
Let $A\in\R^{n\by n}$ be a diagonal matrix with $\det(A)\neq 0$ and let $\sigma\in\inter(\pi_\sut(\Bop(n)))$. Then,
\begin{align}
\label{eq:opt_bop_generic}
\max_{X\in\Bop(n)}\set{\ip{A,X}:\, \pi_\sut(X)=\sigma}
\end{align}
has a unique optimizer $\hat X$. It holds that $\hat X\in\O(n)$. If additionally $\det(A)>0$, then $\hat X\in\SO(n)$.
\end{proposition}

\begin{proof}
First, note that $\Bop(n)$ is full-dimensional in $\R^{n \times n}$ so its projection $\pi_\sut(\Bop(n))$ is also full-dimensional. Thus, $\inter(\pi_\sut(\Bop(n))) = \pi_\sut(\inter(\Bop(n)))$ and \eqref{eq:opt_bop_generic} is strictly feasible. We deduce that strong duality and dual attainability hold in the following primal and dual problems
\begin{align*}
    &\max_{X\in\Bop(n)}\set{\ip{A,X}:\, \pi_\sut(X) = \sigma}\\
    &\qquad=\min_{Y\in\R^{n\by n},\,\lambda\in\R^{\binom{n}{2}}}\set{\ip{\sigma,\lambda} + \norm{Y}_{\tr} :\, Y + \pi_\sut^*(\lambda) = A}.
\end{align*}

Now, let $(\hat Y,\hat\lambda)$ optimize the dual problem. Note that $\hat Y = A - \pi^*_\sut(\hat\lambda)$ is upper triangular with $\diag(A)$ on its diagonal, so $\det(\hat{Y}) = \det(A) \neq 0$ by assumption, i.e., $\rank(\hat Y) = n$.

Let $\hat X\in\Bop(n)$ be an arbitrary maximizer of \eqref{eq:opt_bop_generic}, which exists by compactness of $\Bop(n)$. By strong duality,
\begin{align*}
    \norm{\hat Y}_{\tr} + \ip{\sigma,\hat\lambda} = \ip{A,\hat X} = \ip{\hat Y + \pi_\sut^*(\hat\lambda), \hat X} = \ip{\hat Y,\hat X} +\ip{\sigma,\hat\lambda}.
\end{align*}
Thus, $\norm{\hat Y}_{\tr} = \ip{\hat Y,\hat X}$. Let $U\Sigma V^\intercal = \hat Y$ be an SVD of $\hat Y$. Then, $\tr(\Sigma) = \norm{\hat Y}_{\tr} = \ip{\hat Y, \hat X} = \ip{\Sigma, U^\intercal \hat X V}$. Noting that $U^\intercal \hat X V\in \Bop(n)$ and that $\Sigma$ has only positive diagonal entries, we deduce that $U^\intercal \hat X V = I$ so that $\hat X = UV^\intercal\in \O(n)$. This proves the first claim.

Now, suppose $\det(A) > 0$. Then, $\det(\hat Y) = \det(A - \pi^*_\sut(\hat\lambda)) = \det(A) >0$. Thus, $\det(\hat X) = \det(U)\det(V^\intercal) = \frac{\det(\hat{Y})}{\det(\Sigma)}>0$. We conclude that $\hat X\in\SO(n)$.\ije{\qedhere}
\end{proof}

We may now prove \cref{thm:opt_so_with_sut} in full generality.
\begin{proof}
[Proof of \cref{thm:opt_so_with_sut}]
Let $\hat X$ be an optimizer of \eqref{eq:opt_over_bop_with_sut} and set $\sigma = \pi_\sut(\hat X)$.
Now, let $\epsilon\in(0,1]$ and define $\sigma_\epsilon \coloneqq (1-\epsilon)\sigma$.
\change{As $\Bop(n)$ is full-dimensional, we have that $\sigma_\epsilon\in \pi_{\sut}(\inter(\Bop(n)))$.}
If $\det(A)\neq 0$, then define $A_\epsilon \coloneqq A$. Otherwise, construct $A_\epsilon\in\R^{n\by n}$ by setting each zero diagonal entry of $A$ to $\pm\tfrac{\epsilon}{n}$ in such a way that $\det(A_\epsilon) > 0$, \change{which is possible because $A$ is diagonal and its determinant is the product of its diagonal entries}.
Then, applying \cref{prop:opt_so_with_sut_generic} with $A_\epsilon$ and $\sigma_\epsilon$, there exists $X_\epsilon\in\O(n)$ satisfying
\begin{align}
    \label{eq:X_epsilon_objective_value}
    \ip{A,X_\epsilon} \geq \ip{A_\epsilon,X_\epsilon} - \epsilon \geq (1-\epsilon) \ip{A_\epsilon, \hat X} - \epsilon  \geq (1-\epsilon)\left(\ip{A,\hat X} - \epsilon\right) - \epsilon , 
\end{align}
\change{
where the first and third steps are because $\|A - A_{\epsilon}\|_{\tr} \leq \epsilon$ by our choice of $A_{\epsilon}$ and the fact that $\hat X,X_\epsilon \in \Bop(n)$, and the second step follows as $(1-\epsilon)\hat X$ is feasible in
\begin{align*}
    \max_{X\in\Bop(n)}\set{\ip{A_\epsilon, X}:\, \pi_\sut(X) = \sigma_\epsilon} = \ip{A_\epsilon,X_\epsilon}.
\end{align*}

}

Next, consider a sequence $\set{\epsilon_k}\subseteq(0,1]$ converging to zero and the corresponding sequence $\set{X_{\epsilon_k}}\subseteq \O(n)$. As $\O(n)$ is compact, $\set{X_{\epsilon_k}}$ has a subsequential limit $\tilde X\in\O(n)$. By continuity, we have that $\ip{A,\tilde X} \geq \ip{A,\hat X}$ and $\pi_\sut(\tilde X) = \sigma \in \cC$. We deduce that equality holds between \eqref{eq:opt_over_o_with_sut} and \eqref{eq:opt_over_bop_with_sut}.

Finally, suppose $\det(A)\geq 0$ so that $\det(A_\epsilon)>0$. Then, the sequence $\set{X_{\epsilon_k}}$ lies in $\SO(n)$ so that the subsequential limit $\tilde X$ may also be taken to live in $\SO(n)$.\ije{\qedhere}
\end{proof}

\begin{remark}
\label{rem:recover_so_n_opt}
The proof of \cref{thm:opt_so_with_sut} suggests a numerical method for recovering an optimizer of \eqref{eq:opt_over_o_with_sut} from an optimizer, $\hat X$, of \eqref{eq:opt_over_bop_with_sut}. Let $\epsilon>0$ be some small numerical parameter and let $A_\epsilon,\, \sigma_\epsilon$ be as defined in the proof of \cref{thm:opt_so_with_sut}. Then, the unique maximizer of
\begin{align*}
    \max_{X\in\Bop(n)}\set{\ip{A_\epsilon,X}:\, \pi_\sut(X) = \sigma_\epsilon}
\end{align*}
is guaranteed to lie in $\O(n)$ and is \change{approximately optimal in the sense of \eqref{eq:X_epsilon_objective_value} and and approximately feasible as $\sigma_\epsilon\approx\sigma\in\cC$.}

Alternatively, we may shortcut solving two separate convex optimization problems by preemptively replacing $A$ and $\cC$ with $A_\epsilon$ and $\cC_\epsilon$, in such a way that guarantees $\det(A_\epsilon)\neq 0$ and $\cC_\epsilon\subseteq\inter(\pi_\sut(\Bop(n)))$. With these perturbed sets, \cref{prop:opt_so_with_sut_generic} guarantees that any optimizer of
\begin{align*}
    \max_{X\in\Bop(n)}\set{\ip{A_\epsilon,X}:\, \pi_\sut(X) \in \cC_\epsilon}
\end{align*}
lies in $\O(n)$. Analogous statements hold for the $\SO(n)$ setting.
\end{remark}

\subsection{Applications to low-rank constraints}
\label{subsec:rank_one_coordinate}
The following proposition shows that optimization and feasibility problems over $\SO(n)$ with convex constraints on the values of $\ip{B_i,X}$, where $B_i\in\R^{n\by n}$ are low-rank matrices, can be seen as a special case of SUT constraints after a reparameterization of $\SO(n)$.
\begin{proposition}
    \label{prop:low_rank}
    Let $k \le n-1$.
    Fix $u_1, \dots, u_k \in \R^n$, and also fix $v_1, \dots, v_k \in \R^n$. 
    For $i = 1, \dots, m$, let $B_i = \sum_{j=1}^{k} \beta_{ij} u_{j}v_{j}^\intercal$, for some $\beta_{i,j} \in \R$.
    Let $\cB(X)\coloneqq \begin{pmatrix}
        \ip{B_i,X} \end{pmatrix}_{i\in[m]}$.
    Then $\cB(\SO(n))$ is convex.

\end{proposition}
\cref{thm:opt_so_with_sut} then implies that we may decide feasibility of problems of the form $\cB(\SO(n))\cap \cC$ for compact convex $\cC$ via a simple SDP.
As an example, \cref{prop:low_rank} applies if $m\leq n -1$ and $B_i$ are each rank one as in the situation of \cref{subsec:motivation}.

In order to show this proposition, we will need a lemma.

\begin{lemma}
    \label{prop:rank_one}
Let $\set{u_1,\dots, u_{n-1}}\subseteq\R^n$ and $\set{v_1,\dots, v_{n-1}}\subseteq\R^n$. Then, there exists $U,V\in\SO(n)$ such that for all $i\in\change{[n-1]}$, $v_i^\intercal (V^\intercal X U) u_i$ depends only on the strict upper triangular entries of $X$.
\end{lemma}
\begin{proof}
It follows from the existence of QR decompositions that we may upper triangularize the $\set{u_i}$ with a special orthogonal matrix, i.e., there exists $U\in\SO(n)$ such that $\textrm{supp}(U u_i)\subseteq [1,i]$ for each \change{$i \in [n-1]$}.
Similarly, we may lower triangularize the $\set{v_i}$ with a special orthogonal matrix, i.e., there exists $V\in\SO(n)$ such that $\textrm{supp}(V v_i)\subseteq[i + 1, n]$.
Then
\begin{align*}
    \textrm{supp}(Uu_i v_i^\intercal V^\intercal) \subseteq [1,i]\times [i+1,n] \change{ \text{ for }i \in [n-1]}.
\end{align*}
Thus $(Vv_i)^\intercal X (Uu_i)$ depends only the strictly upper triangular entries of $X$.
\end{proof}

We now show \cref{prop:low_rank}.
\begin{proof}[Proof of \cref{prop:low_rank}]
Note that $\cB(\SO(n))$ is a 
linear image of the set \begin{align}
    \label{eq:ro_example_2}
    \set{\begin{pmatrix}
        v_j^\intercal Xu_j \end{pmatrix}_{j\in[k]}:\, X\in\SO(n)}.
\end{align}
Let $U,V\in\SO(n)$ denote the matrices guaranteed by \cref{prop:rank_one}, then \eqref{eq:ro_example_2} is equivalent to
\begin{align}
    \label{eq:ro_example_3}
    \set{\begin{pmatrix}
        v_j^\intercal (V^\intercal XU)u_j \end{pmatrix}_{j\in[k]}:\, X\in\SO(n)}
\end{align}
by the fact that $\SO(n)= V^\intercal \SO(n)U$.
Finally, by the assumed properties of $U$ and $V$, we have that \eqref{eq:ro_example_3} is a linear image of $\pi_\sut(\SO(n))$ so that $\cB(\SO(n))$ is convex.
\end{proof}

\section{Explicit constructions for elements of $\SO(n)$ with fixed strictly upper triangular entries}
\label{sec:utconstructions}

This section gives full characterizations and explicit constructions for $\pi_\sut^{-1}(\sigma)\cap \SO(n)$ and $\pi_\sut^{-1}(\sigma)\cap \O(n)$, where $\sut$ is the strictly upper triangular coordinates in $\R^{n \times n}$, for $\sigma\in\inter(\pi_\sut(\Bop(n)))$. This will allow us to extend \cref{thm:opt_so_with_sut} to an approximation result in the remaining setting $\det(C)<0$.

We overload notation below. Given $A\in\S^n_+$, let
\begin{align*}
    \O(A) &\coloneqq \set{X\in\R^{n\by n}:\, X^\intercal X = A}\\
    \Bop(A) &\coloneqq \set{X\in\R^{n\by n}:\, X^\intercal X \preceq A}.
\end{align*}
Note that $\O(I_n) = \O(n)$ and $\Bop(I_n) = \Bop(n)$. Furthermore, if $A\in\S^n_{++}$, then
\begin{align*}
    \inter(\Bop(A)) = \set{X\in\R^{n\by n}:\, X^\intercal X \prec A}
\end{align*}
is full-dimensional. Thus, $\inter(\pi_\sut(\Bop(A))) = \pi_\sut(\inter(\Bop(A)))$.

We will require the following technical lemma.
\begin{lemma}
    \label{lem:complete_OA_from_submatrix}
Let $A\in\S^n_{++}$ and $\tilde U\in\R^{n\times (n-1)}$.
Suppose $\tilde U^\intercal \tilde U = A_{2,2}$, the bottom right $(n-1)\by (n-1)$ submatrix of $A$.
Suppose also that the bottom $(n-1)\by(n-1)$ submatrix of $\tilde U$ has full rank.
Then, there exist exactly two choices of $u\in \R^n$ such that
\begin{align*}
    U = \begin{pmatrix}
        u & \tilde U
        \end{pmatrix}\in \O(A)
\end{align*}
and the two choices of $u$ differ on their first coordinates. Furthermore, given $A_{2,2}^{-1}$, the two choices of $u$ can be computed in $O(n^2)$ time.
\end{lemma}
\begin{proof}
Expanding the definition of $U$, we have that $U\in \O(A)$ if and only if
\begin{align*}
    \begin{pmatrix}
    u^\intercal u & u^\intercal \tilde U\\
    \tilde U^\intercal u & \tilde U^\intercal \tilde U
    \end{pmatrix} = \begin{pmatrix}
    A_{1,1} & A_{1,2}\\
    A_{2,1} & A_{2,2}
    \end{pmatrix},
\end{align*}
i.e., if and only if $\norm{u}^2 = A_{1,1}$ and $\tilde U ^\intercal u = A_{2,1}$.

We decompose $\tilde U^\intercal$ as $\tilde U^\intercal = \begin{pmatrix}
    \hat u &
    \hat U^\intercal
    \end{pmatrix}
$,
where $\hat u\in\R^{n-1}$ and $\hat U\in \R^{(n-1)\by (n-1)}$. By assumption, $\hat U$ is invertible.
Thus, $\ker(\tilde U^\intercal)$ is one-dimensional and spanned by the vector
\begin{align*}
    z \coloneqq \begin{pmatrix}
    1\\
    -\hat U^{-\intercal}\hat u
    \end{pmatrix}.
\end{align*}

Next, note that $u_0\coloneqq \tilde U A_{2,2}^{-1} A_{2,1}$ satisfies $\tilde U^\intercal u_0 = A_{2,1}$.
Thus, $u_0 + tz$ parameterizes the solutions of $\tilde U^\intercal u = A_{2,1}$.

Note that $u_0$ has squared norm
\begin{align*}
    \norm{u_0}^2 = A_{2,1}^\intercal A_{2,2}^{-1}(\tilde U^\intercal\tilde U) A_{2,2}^{-1} A_{2,1}=A_{1,2} A_{2,2}^{-1} A_{2,1} < A_{1,1} , 
\end{align*}
where the last inequality follows by the Schur complement lemma and the assumption that $A\in\S^n_{++}$.
We deduce that the quadratic equation $\norm{u_0 + t z}^2 = A_{1,1}$ in $t$ has exactly two solutions.
In other words,
there are exactly two choices of $u\in\R^n$ such that $U\in\O(A)$.
Then, as $z_1 = 1$, we have that the two possible choices of $u$ differ in their first coordinates.

We now turn to the time complexity. Note that $A_{2,2} = \tilde U^\intercal \tilde U = \hat u \hat u^\intercal + \hat U^\intercal \hat U$. Thus, $(\hat U^\intercal \hat U)^{-1} = (A_{2,2} - \hat u\hat u^\intercal)^{-1}$. This quantity can be computed in $O(n^2)$ time given $A_{2,2}^{-1}$ using the Sherman--Morrison formula. Then, the quantity $-\hat U^{-\intercal}\hat u$ can be written as
\begin{align*}
    -\hat U^{-\intercal}\hat u &= - \hat U^{-\intercal}\hat U^{-1} \hat U \hat u\\
    &= - (A_{2,2} - \hat u\hat u^\intercal)^{-1} \hat U\hat u.
\end{align*}
We deduce that the quantities $u_0$ and $z$ can both be computed in $O(n^2)$ time. Finally, computing the two choices of $t$ can also be done within this time limit.
\ije{\qedhere}
\end{proof}

\begin{remark} \label{r:upperTriContinuity}
    The output of the construction in \cref{lem:complete_OA_from_submatrix} is continuous in $\tilde U$ and $A$ wherever it is defined. Formally, there are two continuous functions $u_1$ and $u_2$ from
    \begin{align*}
        \set{(\tilde U, A)\in\R^{n\by  (n-1)}\times \S^n:\, \begin{array}{l}
        A\in\S^n_{++}\\
        \tilde U^\intercal  \tilde U = A_{2,2}\\
        \hat U\text{ is invertible}
        \end{array}}
    \end{align*}
    to $\R^n$ that track the two possible choices of $u$ in \cref{lem:complete_OA_from_submatrix}.
    This follows from continuity of $z$, $u_0$, and the coefficients in the quadratic equation $\norm{u_0 + tz}^2 = A_{1,1}$ in the proof of \cref{lem:complete_OA_from_submatrix}.
\end{remark}

The following \change{proposition} provides a parameterized construction for the entire set $\pi_\sut^{-1}(\sigma)\cap \O(n)$.
\begin{proposition}
    \label{thm:triangle_general_A}
Let $A\in\S^n_{++}$ and $\sigma\in\inter(\pi_\sut(\Bop(A)))$. Then, $\abs{\pi_\sut^{-1}(\sigma)\cap \O(A)}=2^n$.
Furthermore, no two matrices in $\pi_\sut^{-1}(\sigma)\cap \O(A)$ agree on all of their diagonal entries.
\end{proposition}
\begin{proof}
We will induct on $n$. The claim is vacuously true for $n = 1$, thus assume $n\geq 2$.

Let $X\in\inter(\Bop(A))$ satisfy $\sigma=\pi_\sut(X)$.
Partition $X$ and $A$ as
\begin{align*}
    X = \begin{pmatrix}
    \xi & x^\intercal\\
    \bar x & X_{2,2}
    \end{pmatrix},\qquad
    A = \begin{pmatrix}
        \alpha & a^\intercal\\
        a & A_{2,2}
        \end{pmatrix}.
\end{align*}

As $X\in\inter(\Bop(A))$, we have that $A \succ X^\intercal X$ so that \change{its bottom-right blocks are also ordered:}
\begin{align*}
    A_{2,2} \succ xx^\intercal + X_{2,2}^\intercal X_{2,2}.
\end{align*}
\change{Here, we have used the fact that the bottom-right block of $X^\intercal X$ is $xx^\intercal + X_{2,2}^\intercal X_{2,2}$.

Now, subtracting $xx^\intercal$ from both sides, we have that $A_{2,2}-xx^\intercal \succ X_{2,2}^\intercal X_{2,2} \succeq 0$, i.e. $X_{2,2} \in \inter(\Bop(A_{2,2} - xx^\intercal))$ and $A_{2,2}-xx^\intercal$ is positive definite.}

By induction, there exist exactly $2^{n-1}$ matrices $U_{2,2}\in O(A_{2,2}- xx^\intercal)$ matching the strictly upper triangular entries of $X_{2,2}$.
For each of these choices, the matrix $U_{2,2}$ has rank $n-1$.
Define $\tilde U \in\R^{n\times (n-1)}$ as
\begin{align*}
    \tilde U = \begin{pmatrix}
    x^\intercal\\
    U_{2,2}
    \end{pmatrix}.
\end{align*}
Note that $\tilde U^\intercal \tilde U = xx^\intercal + U_{2,2}^\intercal U_{2,2} = A_{2,2}$.
By \cref{lem:complete_OA_from_submatrix}, for each choice of $\tilde U$, there are exactly two ways to append a column to the left of $\tilde U$ to construct a matrix $U\in\O(A)$.
Furthermore, the two choices differ in their diagonal entry.
Finally, we note that the strictly upper triangular entries of $U$ match the strictly upper triangular entries of $X$.\ije{\qedhere}
\end{proof}

For each $\rho\in\set{\pm 1}^n$, we may now define the map
$X_\rho:\inter(\pi_\sut(\Bop(n)))\to\O(n)$ to be the output of the above construction applied to $\sigma\in\inter(\pi_\sut(\Bop(n)))$, where in the $k \by k$ submatrix we pick the larger (or smaller) possible diagonal entry if $\rho_{n-k+1}$ is positive (or negative).
For example, if $\sigma = 0\in\R^{\binom{n}{2}}$, then $X_{\rho}(\sigma) = \Diag(\rho)\in\O(n)$.
Inductively applying \cref{r:upperTriContinuity}, one may verify that $X_\rho(\sigma)$ is continuous as a function of $\sigma \in\inter(\Bop(n))$.

\begin{lemma}
Given $\sigma\in\inter(\pi_\sut(\Bop(n)))$ and $\rho\in\set{\pm 1}^n$, we can construct
$X_\rho(\sigma)$ in time $O(n^3)$.
\end{lemma}
\begin{proof}
We will apply the construction of \cref{thm:triangle_general_A} using \cref{lem:complete_OA_from_submatrix} while recursively maintaining $A_{2,2}^{-1}$ in time $O(n^2)$.
It is clear that we have access to $A_{2,2}^{-1}$ at the very top of the recursion as $A_{2,2}^{-1} = I_{n-1}^{-1} = I_{n-1}$.
It remains to prove the following fact: Given $A\in\S^k$ and $x\in\R^k$ such that $A - xx^\intercal \succ 0$, it is possible to compute the inverse of the bottom-right $(k-1)\by(k-1)$ block of $A - xx^\intercal$ from the inverse of $A$ in time $O(n^2)$.

Write
\begin{align*}
    A - xx^\intercal = \begin{pmatrix}
    \alpha & a^\intercal\\
    a & A_{2,2}
    \end{pmatrix}.
\end{align*}
Note that $\alpha >0$ by the assumption that $A - xx^\intercal\succ 0$. Then,
\begin{align*}
    \begin{pmatrix}
    \alpha &\\&A_{2,2}
    \end{pmatrix} = A - xx^\intercal - \begin{pmatrix}
    0 & a^\intercal\\ a & 0
    \end{pmatrix}.
\end{align*}
Thus, we can compute $A_{2,2}^{-1}$ by computing the inverse of the expression on the right hand side and taking its bottom right block. This can be done via the Sherman--Morrison formula in time $O(n^2)$.
Repeating once for each of the $n$ entries results in $O(n^3)$ time in total.\ije{\qedhere}
\end{proof}

By \cref{thm:triangle_general_A}, $\diag\left(\pi_\sut^{-1}(\sigma)\cap\O(n)\right)\subseteq\R^n$ is a set of $2^n$ distinct elements. 
The following result, due to~\citet{fiedler2009suborthogonality}, allows us to deduce that the $2^n$ elements correspond to the vertices of a (scaled) hypercube.
\begin{lemma}[\citet{fiedler2009suborthogonality}]
\label{thm:fiedler}
Let $U\in\O(n)$ and let $R\in\R^{a \change{\times}b}$ be a submatrix of $U$ where $a + b > n$. Then, $\norm{R}_\textup{op} = 1$.
\end{lemma}

\begin{lemma}
    \label{lem:diagonals_of_so}
Let $\sigma\in\inter(\pi_\sut(\Bop(n)))$.
There exist scalars $\alpha_i < \beta_i$ for $i\in[n]$ (independent of $\rho$) such that for all $\rho\in\set{\pm 1}^n$,
\begin{align*}
    X_{\rho}(\sigma)_{i,i} = \begin{cases}
        \alpha_i & \text{if } \rho_i = -1\\
        \beta_i & \text{if } \rho_i = 1
    \end{cases}.
\end{align*}
\end{lemma}
\begin{proof}
    Fix $\sigma\in\inter(\pi_\sut(\Bop(n)))$ and let $X\in\inter(\Bop(n))$ such that $\pi_\sut(X) = \sigma$.
    
    For $i\in[n]$, let $\hat R_i$ denote the $i\by (n-i+1)$ dimensional submatrix of $X$ with bottom-left entry at coordinate $(i,i)$. 
    Note that $\norm{\hat{R}_{i}}_\textup{op} \leq \norm{X}_\textup{op} < 1$ as $X \in \inter(\Bop(n))$. 
Let $R_i(s)\in\R^{i\by(n-i+1)}$ denote the matrix that replaces the bottom-left entry of $\hat R_i$ with $s\in\R$.
    Then, $R_i(s)$ parameterizes a line that intersects the interior of the operator norm ball \changetwo{(which we will denote here by $\Bop(i,n-i+1)$)}.
    As the operator norm ball is a compact convex body, there are exactly two choices of $s$, denoted $\alpha_i<\beta_i$, for which $\norm{R_i(s)}_\textup{op} = 1$.

Then, by \cref{thm:fiedler}, we deduce that $\diag\left(\pi_\sut^{-1}(\sigma) \cap \O(n)\right) \subseteq \prod_i \set{\alpha_i,\beta_i}$.
    Combining with \cref{thm:triangle_general_A} completes the proof.
    \ije{\qedhere}
\end{proof}

The following result states that the sign of $\det(X_{\rho}(\sigma))$ depends only on $\rho$.
\begin{lemma}
    \label{thm:sign_of_x_rho}
Let $\sigma\in\inter(\pi_\sut(\Bop(n)))$ and $\rho\in\set{\pm 1}^n$. Then $\det(X_{\rho}(\sigma)) = \prod_i \rho_i$.
\end{lemma}
\begin{proof}
Fix $\rho\in\set{\pm 1}$ and $\sigma\in\inter(\pi_\sut(\Bop(n)))$ and consider the continuous function $f(\alpha)\coloneqq \det(X_{\rho}(\alpha\sigma))$ defined on $\alpha\in[0,1]$. As $X_{\rho}(\alpha\sigma)\in\O(n)$ for all $\alpha\in[0,1]$, we have that $f$ can only take on the values $\pm 1$. As $f$ is also continuous, it must be constant so that $f(1) = f(0) = \det(X_{\rho}(0)) = \det(\Diag(\rho)) = \prod_i \rho_i$, where we used $X_{\rho}(0) = \Diag(\rho)$ as mentioned above.\ije{\qedhere}
\end{proof}

\subsection{Refinements of \cref{thm:opt_so_with_sut}}

The following theorem extends \cref{thm:opt_so_with_sut} to an approximation result in the remaining case to maximization over $\SO(n)$ with SUT constraints and $\det(A)<0$.
\begin{theorem}
\label{thm:approximate_det_c_negative_generic}
Let $A\in\R^{n\by n}$ be a diagonal matrix with $\det(A)<0$ and let $\cC\subseteq \pi_\sut(\Bop(n))$ be a nonempty closed convex set. Then \eqref{eq:opt_over_bop_with_sut} provides a $\left(1 - \frac{1}{n}\right)$-approximation of \eqref{eq:opt_over_so_with_sut} in the following sense:
\begin{align*}
    &\max_{X\in\SO(n)}\set{\ip{A,X}:\, \pi_\cT(X)\in\cC}\\
    &\qquad\geq \left(1 - \frac{1}{n}\right) \max_{X\in\Bop(n)}\set{\ip{A,X}:\, \pi_\cT(X) \in\cC} + \frac{1}{n}\min_{X\in\Bop(n)}\set{\ip{A,X}:\, \pi_\cT(X) \in\cC}.
\end{align*}
\end{theorem}
\begin{proof}
Let $\hat\rho = \sign(\diag(A))$ and $\hat X := \hat{X}_{\hat{\rho}} \in\Bop(n)$ maximize \eqref{eq:opt_over_bop_with_sut} with $\sigma = \pi_\sut(\hat X)$.
We will only consider the case where $\sigma\in\inter(\pi_\sut(\Bop(n)))$. The general case follows by continuity and compactness.

We will fix $\sigma$ in the remainder of the proof and write $X_\rho$ instead of $X_\rho(\sigma)$.
Let $(\alpha_i,\beta_i)$ be the quantities furnished by \cref{lem:diagonals_of_so} applied to $\sigma$.
For $i\in[n]$, let $\rho^{(i)}\in\set{\pm 1}^n$ denote the vector that negates the $i$th entry of $\hat{\rho} = \sign(\diag(A))$, and by \cref{thm:sign_of_x_rho}, for all $i\in[n]$, we have
$X_{\rho^{(i)}}\in\SO(n)$ and $\pi_\sut(X_{\rho^{(i)}}) = \sigma \in\cC$. 

Then,
\begin{align*}
    \max_{X\in\SO(n)}\set{\ip{A,X}:\, \pi_\sut(X) \in\cC}
    &\geq \max_{i\in[n]} \ip{A, X_{\rho^{(i)}}}\\
&= \ip{A,\hat{X}_{\hat{\rho}}} - \min_{i\in[n]}\abs{A_{i,i}(\beta_i - \alpha_i)}\\
    &\geq \ip{A,\hat{X}_{\hat{\rho}}} - \frac{1}{n}\left(\sum_{i=1}^n \abs{A_{i,i}(\beta_i - \alpha_i)}\right)\\
    &= \ip{A,X_{\hat\rho}} - \frac{1}{n}\left(\ip{A, X_{\hat\rho}} - \ip{A,X_{-\hat\rho}}\right)\\
    &= \left(1 - \frac{1}{n}\right)\ip{A,X_{\hat\rho}} + \frac{1}{n}\ip{A, X_{-\hat\rho}},
\end{align*}
where in the second line we used that $\diag(X_{\rho^{i}})$ differs from $\diag(\hat{X}) = \diag(\hat{X}_{\hat{\rho}})$ only in the $i$-th entry, swapping $\alpha_{i}, \beta_{i}$,
and in the fourth line we used that $\hat{\rho} = \sign(\diag(A))$, so $\hat{X}_{\pm \hat{\rho}}$ satisfy $\pi_\sut(X_{\pm \hat\rho})=\sigma\in\cC$ for $\sigma\in\inter(\pi_\sut(\Bop(n)))$, and $\hat{X}_{\pm \hat{\rho}}$ are the maximizer and minimizer of the optimization problem \cref{eq:opt_over_o_with_sut} over $\O(n)$.\ije{\qedhere}
\end{proof}

\change{The approximation factor of $1-\tfrac{1}{n}$ in the statement of \cref{thm:approximate_det_c_negative_generic} is optimal as can be seen by considering the case $A= \Diag(1,1,\dots,1,-1)$ and $\cC= \set{\pi_\sut(0_{n\by n})}$. In this case, the optimal value of the optimization problem over $\SO(n)$ is $n-2$, whereas the maximum and minimum values over $\O(n)$ are $n$ and $-n$ respectively.}

\section{Obstructions to further generalization}
\label{sec:obstructions}

This section collects a number of results showing that our hidden convexity results are essentially tight.

\subsection{Maximality of \cref{thm:twoconvex}}
Recall that \Cref{thm:twoconvex} shows any two dimensional linear image of $\SO(n)$ is convex. 
The following result shows this
is \change{maximal} in a specific sense.
\begin{lemma}
    \label{thm:nonconvex}
    For any $n\geq 3$, there exists a linear map $\pi : \SO(n) \rightarrow \R^3$ so that $\pi(\SO(n))$ is nonconvex.
\end{lemma}
\begin{proof}
    We define 
    \[
        \pi(X) = \left(X_{11}, X_{12}, \sum_{i=3}^n X_{ii}\right).
    \]
    Let $H = \{x \in \R^3 : x_{3} = n-2\}$.

    To see that $S = \pi(\SO(n))$ is nonconvex, we show that $S \cap H$ is nonconvex.
    Consider a general $X \in \SO(n)$. It holds that $\sum_{i=3}^n X_{ii} = n-2$ if and only if $X_{ii} = 1$ for all $i > 2$.
    This occurs if and only if $X$ is block diagonal, so that
    \[
        X = \begin{pmatrix}
            \begin{smallmatrix}
            \cos(\theta) & \sin(\theta)\\
            -\sin(\theta) & \cos(\theta)
            \end{smallmatrix} & \\
            & I
        \end{pmatrix},
    \]
    for some $\theta \in [0,2\pi]$.

    Now, if $X$ has this form, then $\pi(X) = (\cos(\theta), \sin(\theta), n-2)$, i.e.,
    \[
        S \cap H = \{(\sin(\theta), \cos(\theta), n-2) : \theta \in [0,2\pi]\}.
    \]
    This is clearly nonconvex, for example, $(0,0,n-2) \in \conv (S \cap H) \setminus (S \cap H)$.\ije{\qedhere}
\end{proof}

In fact, this construction gives us an example of a 2-constraint optimization problem over $\SO(n)$ for which the $\conv(\SO(n))$ relaxation is not tight.
Consider the following optimization problem:
\begin{align}
    \label{eq:two_constraint_noncvx}
        \max_{X\in \SO(n)} \left\{\sum_{i=3}^n X_{ii} : \begin{array}{l} X_{1,1} = 0\\  X_{1,2} = 0\end{array}\right\}.
\end{align}
We have seen that it is not possible for a matrix in $\SO(n)$ to attain a value of $n-2$ in this problem, since any matrix in $\SO(n)$ where $\sum_{i=3}^n X_{ii} = n-2$ has the property that $X_{11}^2+X_{12}^2 = 1$.
However, $n-2$ is attainable by a convex combination of matrices in $\SO(n)$,
\[
    \frac{1}{2}\left(
        \begin{pmatrix}
        \begin{smallmatrix}
        1 & 0\\
        0 & 1
        \end{smallmatrix} & \\
        & I
        \end{pmatrix}
    +
        \begin{pmatrix}
        \begin{smallmatrix}
        -1 & 0\\
        0 & -1
        \end{smallmatrix} & \\
        & I
        \end{pmatrix}
    \right).
\]
Thus, the convex relaxation of \eqref{eq:two_constraint_noncvx} that replaces $\SO(n)$ with $\conv(\SO(n))$ achieves value $n - 2$.

\subsection{Maximality of \cite[Theorem 8]{horn1954doubly}}
Horn's result \cite[Theorem 8]{horn1954doubly} shows that the diagonal projection of $\SO(n)$ is a convex polytope.
A slight modification of the construction from the previous subsection shows this is maximally convex in the following sense:
\begin{lemma}
    \label{thm:diagonal_maximal}
    Let $A \in \R^{n\times n}$ be a nondiagonal matrix.
    Consider the linear map $\pi_A : \R^{n\times n} \rightarrow \R^{n+1}$, so that $\pi_A(X)_{i} = X_{ii}$ for $i = 1, \dots, n$, and $\pi_A(X)_{n+1} = \ip{A,X}$.
    Then $\pi_A(\SO(n))$ is not convex.
\end{lemma}
\begin{proof}
    We first consider the case when $A$ is not symmetric.
    By permuting coordinates, we may assume $A_{1,2} \neq A_{2,1}$.
    Define $H = \{x \in \R^{n+1} :\, x_{i} = 1 \text{ for } i = 3, \dots, n\}$ and consider $\pi(\SO(n)) \cap H$.

    We have seen that if $X \in \SO(n)$ and $X_{ii} = 1$ for $i = 3, \dots, n$, then 
    \[
        X = \begin{pmatrix}
            \begin{smallmatrix}
            \cos(\theta) & \sin(\theta)\\
            -\sin(\theta) & \cos(\theta)
            \end{smallmatrix} & \\
            & I
        \end{pmatrix},
    \]
    and therefore,
    \[
        \ip{A, X} = (A_{11} + A_{22})\cos(\theta) + (A_{1, 2} - A_{2,1})\sin(\theta) + \sum_{i=3}^n A_{ii}.
    \]
    We will consider the linear map of $\pi(\SO(n))\cap H$ into $\R^2$ that sends $\pi(X)$ to
    \[
        \left(X_{1,1}, \frac{\ip{A,X} - (A_{11} + A_{22})X_{1,1} - \sum_{i=3}^n A_{ii}}{A_{1,2}-A_{2,1}}\right) = (\cos(\theta), \sin(\theta)).
    \]
    In other words, this linear map sends $\pi(\SO(n)) \cap H$ to a circle. We conclude that $\pi(\SO(n))$ is not convex.

    Now, we consider the case when $A$ is symmetric but not diagonal.
    We may permute coordinates to assume that $A_{1,2} = A_{2,1}\neq 0$.
    Let $D^{(i)}$ be a diagonal matrix where $D^{(i)}_{ii} = -1$ and for all $j \neq i$, $D^{(i)}_{jj} = 1$.
    Then, define $A' = D^{(1)} A D^{(2)}$, which is not symmetric. Note that
    \[
        \pi_{A'}(X) = (X_{11}, X_{22}, X_{33},\dots, X_{nn}, \langle A', X\rangle) = \tau(\pi_A(D^{(1)}XD^{(2)})),
    \]
    where $\tau(x) = (-x_1, -x_2,x_3 ,\dots, x_n, x_{n+1})$. In particular, $\pi_A(\SO(n))$ is convex if and only if $\pi_{A'}(\SO(n))$ is convex, from which the claim follows.\ije{\qedhere}
\end{proof}

\subsection{Maximality of \cref{thm:feasibility}}
\cref{thm:feasibility} gives an example of an $m = \binom{n}{2}$-dimensional linear projection of $\SO(n)$ that is convex. 
The following lemma shows this is maximal in terms of dimension.

\begin{lemma}
    Suppose $n\geq 3$ and $\pi : \R^{n\times n} \rightarrow \R^m$ satisfies $\rank(\pi) > \binom{n}{2}$.
    Then, $\pi(\SO(n))$ is non-convex.
\end{lemma}
\begin{proof}
    It suffices to show this statement in the case where $\rank(\pi)=m$.
    Suppose that $\pi(\SO(n)) = \pi(\conv(\SO(n)))$ is convex. 
    A convex set has the property that it either contains an interior point or is contained in a proper affine subspace.
    As $\conv(\SO(n))$ is full-dimensional for all $n\geq 3$ and $\pi$ has full rank, we deduce that $\pi(\conv(\SO(n)))$ cannot be contained in a proper affine subspace of $\R^m$. 

Therefore, $\pi(\SO(n))$ must have an interior point and in particular, has a positive measure. 

    This is a contradiction to Sard's lemma: $\SO(n)$ is $\binom{n}{2}$-dimensional and $\R^m$ is $m$-dimensional, but Sard's lemma states that the image of a manifold of dimension less than $m$ under a smooth map into $\R^m$ must have measure zero.\ije{\qedhere}
\end{proof}

\subsection{Necessity of $\det(A)\geq 0$ in \cref{thm:opt_so_with_sut}}
We have shown that if $A$ is a diagonal matrix with $\det(A)\geq 0$, then the optimization problem
\begin{equation*}
        \max_{X \in \SO(n)} \{\ip{A,X} :  \pi_\sut(X) = \sigma\}
\end{equation*}
agrees with the convex relaxation replacing $\SO(n)$ with $\Bop(n)$ (see \cref{thm:opt_so_with_sut}).
The following numerical example shows that the assumption that $\det(A)\geq 0$ cannot be dropped in \cref{thm:opt_so_with_sut} even if we strengthen the convex relaxation by replacing $\SO(n)$ with $\conv(\SO(n))$.

The following numerical example is computed using the 
cvxpy convex optimization package \cite{diamond2016cvxpy}
and the description of $\conv(\SO(3))$ given in \cite{saunderson2015semidefinite}[Theorem 1.3]:
\begin{equation*}
    \max_{X\in\SO(3)} \left\{X_{1,1}+X_{2,2}-X_{3,3} :
    \begin{array}{l}
        X_{1,2} = 0.5\\
        X_{1,3} = 0.3\\
        X_{2,3} = 0.2
    \end{array}
        \right\} = 0.921,
\end{equation*}
whereas
\begin{equation*}
    \max_{X\in\conv(\SO(3))} \left\{X_{1,1}+X_{2,2}-X_{3,3} :
    \begin{array}{l}
        X_{1,2} = 0.5\\
        X_{1,3} = 0.3\\
        X_{2,3} = 0.2
    \end{array}
        \right\} = 1.0.
\end{equation*}

\section{Summary and open questions}
In this paper, we proved new hidden convexity results inspired by solving constrained optimization problems over $\SO(n)$ and $\O(n)$. These results in turn show that specific structured instances of constrained optimization over $\SO(n)$ and $\O(n)$ can be efficiently solved via their convex relaxations. We close by posing some natural questions surrounding hidden convexity.

\paragraph{Convex coordinate projections.}
In general, it seems to be difficult to fully characterize the possible sets of coordinates $S \subseteq [n] \times [n]$ so that the projection of $\SO(n)$ onto the coordinates in $S$ is convex.

We will note some basic invariants of this question: clearly if $\pi_S(\SO(n))$ is convex, then for all $T \subseteq S$, \change{$\pi_T(\SO(n))$} is convex.
We say that $S$ has a property \emph{up to permutation} if there are permutations $\sigma$ and $\rho$ so that $\{(\sigma_i, \rho_j) : (i, j) \in S\}$ has this property.
Similarly we say that $S$ has a property \emph{up to transposition} if either $S$ or $\{(j,i) : (i, j) \in S\}$ has this property.
Clearly, $S$ has the property that $\pi_S(\SO(n))$ is convex if and only if it has this property up to permutations and transposition.

Note also that by \cref{thm:fiedler}, that $\pi_S(\SO(n))$ is not convex if $S$ contains a rectangle of size $a\times b$ where $a+b>n$. Here, by rectangle we mean a subset of $S$ of the form $A\times B\subseteq S$ where $A,\, B\subseteq [n]$ and $\abs{A} = a$ and $\abs{B} = b$.

Using this idea with additional casework (which we feel is ultimately uninformative)
we can obtain the following characterization of the coordinate subsets of $[4] \times [4]$ so that $\pi_S(\SO(4))$ is convex:
\begin{lemma}
    Let $S \subseteq [4] \times [4]$ be such that $\pi_S(\SO(4))$ is convex. Then (up to permutations and transpositions), $S$ is a subset of one of the following: 
    \begin{itemize}
        \item $T = \{(i,j) : i < j \in [4]\}$
        \item $D = \{(i,i) : i \in [4]\}$
        \item $F = \{(1,1),(1,2),(2,3),(2,4)\}$.
    \end{itemize}
\end{lemma}

The structure of these examples suggest that there may be some rich combinatorial information hidden in the question of whether or not a given coordinate projection of $\SO(n)$ is convex.
In particular, we suspect the following: consider the decision problem \texttt{CONVEX} whose input is a set $S \subseteq [n] \times [n]$, and whose output is TRUE if $\pi_S(\SO(n))$  is convex, and FALSE otherwise.
\begin{conjecture}
    The problem \texttt{CONVEX} is NP-hard.
\end{conjecture}
We will remark that it is not even clear if this problem is in NP, as there does not seem to be an efficient witness to the fact that $\pi_S(\SO(n))$ is convex. We note that determining whether $S \subseteq [n] \times [n]$ is (up to permutations and transpositions) a subset of the upper triangular entries of $[n] \times [n]$ is NP-hard \cite{fertin2015obtaining}.

\paragraph{Small semidefinite representation of two-dimensional images.}
It is known that the smallest equivariant (see \cite{saunderson2015semidefinite} for a definition) semidefinite representation of $\conv(\SO(n))$ is exponential in size~\cite{saunderson2015semidefinite}. We leave open the question of whether $\pi(\SO(n))$, where $\pi:\R^{n\by n}\to\R^2$, may have a small (possibly linear-sized) semidefinite representation.

\paragraph{Hidden convexity of multiple copies of $\SO(n)$.}
Finally, we also leave the study of convex images of direct products of $\SO(n)$ to future work. Such results may be useful in applications such as cryo-EM~\cite{bandeira2020non}, where the optimization problems contain multiple $\SO(n)$ matrices.

\section*{Acknowledgments}
Akshay Ramachandran was supported by the European Research Council (ERC)
under the European Union’s Horizon 2020 research and innovation programme:
AR from grant agreement no.\ 805241-QIP.
Kevin Shu was supported by the ACO-ARC fellowship while conducting this work.
Part of this work was completed while Alex L.\ Wang was supported by the Dutch Scientific Council (NWO) grant OCENW.GROOT.2019.015 (OPTIMAL).
 
\newpage
{
        \bibliographystyle{plainnat}

}

\begin{appendix}
\crefalias{section}{appendix}

\section{Separation and optimization oracles for the parity polytope}
\label{app:separation_pp}

    It is possible to implement separation and optimization oracles for $\PP_n$ that run in $\O(n\log n)$ time.

    \paragraph{Separation.}
    We will use the following description of $\PP_n$ given in \cite{jeroslow1975defining,Lancia2018}:
\[ \PP_{n} \coloneqq \set{ x \in [-1,1]^{n} :\, \ip{x, 1_{n} - 2 \cdot 1_{S}} \leq n-2 ,\,\quad \forall \text{ odd  } S \subseteq [n]}.     \]
    Here, we will say that $S\subseteq[n]$ is odd if $\abs{S}$ is odd. Else, $S$ is even.
    The set of constraints can be rewritten as
    \[ \min_{\text{odd }S\subseteq[n]} \ip{x, 1_{S}} \geq \frac{1}{2} (\langle x, 1_{n} \rangle - (n-2) ) .  \]
In $O(n\log n)$ time, we may sort the entries of $x$ and compute the sums of all odd-length prefixes of the sorted vector.
    If every sum is at least $\frac{1}{2} (\langle x, 1_{n} \rangle - (n-2) )$, then $x\in\PP_n$. Otherwise, we have found a separating hyperplane.

    \paragraph{Optimization.}
    We will use the vertex description
    \[ \PP_{n} \coloneqq \conv\set{ 1_{n} - 2\cdot 1_{S} :\,
    \text{even } S\subseteq[n]}.  \]
    Then, given $w\in\R^n$, we may optimize $\max_{x \in \PP_{n}} \ip{w,x}$ by solving $\min_{\text{even} S\subseteq[n]} \ip{w,1_S}$. We can construct a minimizer of the latter problem in $O(n\log n)$ time by sorting $w$ and computing the even-length prefix sums of the sorted vector.  

\section{Connections with quadratic convexity theorems}
\label{sec:quadratic_convexity}

This appendix interprets hidden convexity results on $\SO(n)$ as quadratic convexity results on the unit sphere.

A basic result in the Lie group theory of $\SO(n)$ is the existence of a quadratic map $Q:\R^{2^{n-1}}\to\R^{n\by n}$ and a subset $\spin(n)$ of the unit sphere in $\R^{2^{n-1}}$ such that $Q(\spin(n)) = \SO(n)$.
This map is quadratic in the sense that there exists a collection of $n^2$ symmetric matrices $\set{A_{ij}}\subseteq\S^{2^{n-1}}$ indexed by $(i,j)\in[n]^2$ such that
\begin{align*}
(Q(x))_{i,j} = \ip{x, A_{ij} x}.
\end{align*}

This result and its construction are explained in detail in \cite[Appendix A]{saunderson2015semidefinite}.
It is additionally shown in \cite[Theorem 1.1]{saunderson2015semidefinite} that for any $Y\in\R^{n\by n}$,
\begin{gather}
    \label{eq:saunderson_one_dim}
\max_{X\in\SO(n)}\ip{Y, X} = 
\max_{x\in\spin(n)}\ip{Y,Q(x)}
=
\max_{x\in\bS^{2^{n-1}-1}}\ip{Y, Q(x)}.
\end{gather}

Now, let $\pi:\R^{n\by n}\to\R^m$ be a linear function. Then,
\begin{align*}
\pi(\SO(n)) &= (\pi\circ Q)(\spin(n))
\subseteq (\pi\circ Q)(\bS^{2^{n-1}-1})
\subseteq \conv(\pi(\SO(n))).
\end{align*}
Here, the last inclusion follows by \eqref{eq:saunderson_one_dim}.

We deduce that if $\pi(\SO(n))$ is convex, then equality holds throughout this chain and the image of the unit sphere $\bS^{2^{n-1}-1}$ under the quadratic map $\pi\circ Q$ is convex.
For example, combined with \cref{thm:feasibility}, we have that
\begin{align*}
    \set{(\ip{x, A_{ij}x})_{i<j}:\, x\in\bS^{2^{n-1}-1}}\subseteq\R^{\binom{n}{2}}
\end{align*}
is convex.
As another example, combined with \cite[Theorem 8]{horn1954doubly}, we have that
\begin{align*}
    \set{\begin{pmatrix}
    \ip{x,A_{11}x}\\
    \vdots\\
    \ip{x,A_{nn}x}
    \end{pmatrix}:\, x\in\bS^{2^{n-1}-1}}\subseteq\R^n
\end{align*}
is convex (and equal to the polytope $\PP_n$).

 \end{appendix}

\end{document}